\documentclass[final]{siamltex}
\usepackage{fullpage}
\usepackage{amsfonts}
\usepackage{latexsym}
\usepackage{amssymb}
\usepackage{color}
\usepackage{hyperref}
\usepackage{graphicx}
\graphicspath{{./figures/}}
\usepackage{epstopdf}
\usepackage{algorithm,algorithmic}
\usepackage{verbatim}
\usepackage{mathtools}
\usepackage{enumitem}
\usepackage{esvect}
\usepackage{centernot}
\usepackage{float} 
\usepackage{siunitx}
\usepackage{mathrsfs}

\def\a{\alpha}

\def\e{\epsilon}
\def\d{\delta}
\def\g{\gamma}

\def\o{\omega}
\def\s{\sigma}
\def\n{\nabla}
\def\wh{\widehat}
\def\wt{\widetilde}
\def\tsy{\textstyle}
\newcommand{\tsum}{\textstyle{\sum}}
\def\sm{\textstyle{\tfrac{1}{m}\sum_{i=1}^m}}
\def\smn{\textstyle{\sum}_{i=1}^m}
\def\sk{\textstyle{\sum}_{t=1}^s}
\def\skt{\textstyle{\sum}_{t=2}^s}
\def\sl{\textstyle{\sum}_{\ell=1}^k}
\def\sll{\textstyle{\sum}_{\ell=2}^k}
\def\slk{\textstyle{\sum}_{\ell=1}^{k-1}}

\def\R{\mathbb{R}}
\def\M{\mathcal{M}}
\def\O{\mathcal{O}}
\def\la{\langle}
\def\ra{\rangle}
\def\lla{\left\langle}
\def\rra{\right\rangle}
\def\ux{\underline{x}}

\def\bx{\bar{x}}
\def\Bx{{\textbf{x}}}
\def\bBx{\bar{\textbf{x}}}
\def\eBx{\emph{\textbf{x}}}
\def\by{\bar{y}}
\def\l{\ell}
\def\lbd{\lambda}
\def\Eb{\mathbb{E}}

\def\bfA{\mathbf{A}}
\def\bfb{\mathbf{b}}

\renewcommand{\theequation}{\thesection.\arabic{equation}}
\DeclareMathOperator*{\argmin}{arg\,min}
\DeclareMathOperator*{\argmax}{arg\,max}

\newcommand\myeqa{\stackrel{\mathclap{\scriptsize\mbox{(a)}}}{=}}
\newcommand\myeqb{\stackrel{\mathclap{\scriptsize\mbox{(b)}}}{=}}
\newcommand\myeqc{\stackrel{\mathclap{\scriptsize\mbox{(c)}}}{=}}

\newcommand\myleqa{\stackrel{\mathclap{\scriptsize\mbox{(a)}}}{\le}}
\newcommand\myleqb{\stackrel{\mathclap{\scriptsize\mbox{(b)}}}{\le}}
\newcommand\myleqc{\stackrel{\mathclap{\scriptsize\mbox{(c)}}}{\le}}

\newcommand\mygeqa{\stackrel{\mathclap{\scriptsize\mbox{(a)}}}{\ge}}
\newcommand\mygeqb{\stackrel{\mathclap{\scriptsize\mbox{(b)}}}{\ge}}
\newcommand\mygeqc{\stackrel{\mathclap{\scriptsize\mbox{(c)}}}{\ge}}

\newcommand{\beq}{\begin{equation}}
\newcommand{\eeq}{\end{equation}}
\newcommand{\nn}{\nonumber}

\allowdisplaybreaks
\def\vgap{\vspace*{.1in}}
\newtheorem{remark}[theorem]{Remark}

\title{ Accelerated  Stochastic Algorithms for Nonconvex\\
 Finite-sum and Multi-block  Optimization\thanks{
This work was funded by National Science Foundation grants CMMI-1637473, CMMI-1637474 and Army Research Office W911NF-18-1-0223.}}

\author{
	Guanghui Lan
	\thanks{H. Milton Stewart School of Industrial \& Systems Engineering, Georgia Institute of Technology, Atlanta, GA, 30332 .
		(email: {\tt george.lan@isye.gatech.edu}).}
	\and
	Yu Yang
	\thanks{H. Milton Stewart School of Industrial \& Systems Engineering, Georgia Institute of Technology, Atlanta, GA, 30332 .
		(email: {\tt yangyu@gatech.edu}).} 
}

\begin{document}
\maketitle  
\begin{abstract}
		In this paper, we present new stochastic methods for solving two important classes of nonconvex optimization
		problems. We first introduce a randomized accelerated proximal gradient (RapGrad) method for 
		solving a class of nonconvex optimization problems whose objective function consists of the summation of $m$ components, 
		and show that it can significantly
		reduce the number of gradient computations especially when the condition number $L/\mu$ (i.e.,
		the ratio between the Lipschitz constant and negative curvature) is large. More specifically,
		RapGrad can save up to ${\cal O}(\sqrt{m})$ gradient computations than existing batch nonconvex accelerated gradient methods.
		Moreover, the number of gradient computations required by RapGrad can be ${\cal O}(m^\frac{1}{6} L^\frac{1}{2} / \mu^\frac{1}{2})$ (at least ${\cal O}(m^\frac{2}{3})$) 
		times smaller than the best-known randomized nonconvex gradient methods when $L/\mu \ge m$.
		Inspired by RapGrad, we also develop a new randomized accelerated proximal dual (RapDual) method
		for solving a  class of multi-block nonconvex optimization problems coupled with linear constraints and some special structural 
		properties. 
		We demonstrate that RapDual can also save up to a factor of ${\cal O}(\sqrt{m})$ block updates than its
		batch counterpart, where $m$ denotes the number of blocks.
		To the best of our knowledge, all these complexity results associated with RapGrad and RapDual seem
		to be new in the literature. We also illustrate potential advantages of these algorithms through our preliminary numerical experiments.
		
		\vspace{.1in}
		
		\noindent {\bf Keywords:} nonconvex optimization, stochastic algorithms, acceleration, finite-sum optimization, multi-block optimization.
	\end{abstract}	 
	\vspace{.07in}

\noindent {\bf AMS 2000 subject classification:} 90C25, 90C06, 90C22, 49M37
\section{Introduction}   
Nonconvex optimization plays a fundamental role in modern statistics and machine learning, e.g.,
for empirical risk minimization with either nonconvex loss (\cite{tibshirani2013introduction}) or regularization 
(\cite{fan2001variable,zhang2010nearly,zhang2012general}), as well as the training of deep neural networks (\cite{Goodfellow-et-al-2016}). 
In this paper, we consider two classes of nonconvex optimization problems that are widely used in
statistical learning.
The first class of problems intends to minimize the summation of many terms:
\beq\label{problem_non}
\min_{x\in X} \{ f(x):=\sm f_i(x)\},
\eeq
where $X\subseteq\R^n$ is a closed convex set, and
$f_i:X \to \mathbb{R}$, $i=1,\ldots,m$, 
are nonconvex smooth functions with $L$-Lipschitz
continuous gradients over $X$, i.e., for some $L\geq 0$,
\beq
\|\n f_i(x_1)-\n f_i(x_2)\|\leq L\|x_1-x_2\|, \quad \forall x_1, x_2\in X.\label{assumptions_non1}
\eeq
Moreover, we assume that there exists $0<\mu\leq L$ such that (s.t.)
\beq
f_i(x_1)-f_i(x_2)-\la\n f_i(x_2),x_1-x_2\ra \geq -\tfrac{\mu}{2}\|x_1-x_2\|^2, \quad \forall x_1, x_2\in X.\label{assumptions_non2}
\eeq
Clearly, \eqref{assumptions_non1} implies \eqref{assumptions_non2} (with $\mu = L$).
While in the classical nonlinear programming setting one only assumes \eqref{assumptions_non1},
by using both conditions \eqref{assumptions_non1} and \eqref{assumptions_non2} we can explore more structural
information for the design of solution methods of problem~\eqref{problem_non}. In particular, we intend to develop more 
efficient algorithms to solve  problems where the condition number $L/\mu$ associated with problem \eqref{problem_non} is large.
As an example, consider the nonconvex composite problem arising from variable selection in statistics \cite{fan2001variable,ghadimi2016accelerated}: 
$
f(x) = \sm h_i(x) + \rho p(x),
$
where $h_i$'s are smooth convex functions, $p$
is a nonconvex function, and $\rho > 0$ is a relatively small penalty parameter. 
Note that some examples of the nonconvex penalties are given by
minimax concave penalty (MCP) or smoothly clipped absolute deviation (SCAD) (see  \cite{fan2001variable}). It can be shown that
the condition number for these problems is usually larger than $m$ (see Section~\ref{sec-num} for more details).
%

In addition to \eqref{problem_non}, we consider an important class of nonconvex multi-block optimization problems with linearly coupled constraints, i.e.,
\begin{align}\label{c:problem_non}
\min_{x_i\in X_i} &~ \smn f_i(x_i)\nn\\
\text{s.t. } &~\smn A_ix_i = b.
\end{align}
Here $X_i\subseteq\R^{d_i}$ are closed convex sets, $A_i \subseteq \R^{n\times d_i}$, $b \subseteq \R^n$, $f_i:X_i \to \mathbb{R}$
satisfy, for some $\mu\geq0$,
\begin{align}\label{c:assumptions_non1}
f_i(x)-f_i(y)-\la\nabla f_i(y),x-y\ra &\geq -\tfrac{\mu}{2}\|x-y\|^2, \quad \forall x, y\in X_i,
\end{align}
and $f_m:\R^n \to \mathbb{R}$ has $L$-Lipschitz continuous gradients, i.e.,  $\exists L\geq 0$ s.t.
\beq
\|\nabla f_m(x)-\nabla f_m(y)\| \leq  L\|x-y\|, \quad \forall x, y\in \R^n.\label{c:assumptions_non2}
\eeq
Moreover, we assume that $X_m=\R^{n}$ and $A_{m}$ is invertible. 
In other words,  throughout this paper we make the structural assumption that one of the blocks equals the dimension of the variable.
Problem of this type arises 
naturally in compressed sensing and distributed optimization.  For instance, consider the 
compressed sensing problem via nonconvex shrinkage penalties:
$\min_{x_i\in X_i} ~ \left\{p(x):
Ax = b\right\},
$
where $A \in \R^{n \times d}$ is a big sensing matrix with $d >> n$, and $p(x)=\smn p_i(x_i)$ is a nonconvex and separable penalty function.
Since it is easy to find an invertible submatrix in $A$, w.l.o.g, we assume that the last $n$ columns of $A$ forms  an invertible matrix. 
We can then view this problem as a special case of \eqref{c:problem_non} by grouping the last $n$ components of $x$ into block $x_m$, 
and dividing the remaining $d-n$ components into another $m-1$ blocks. 

 Much recent research effort has been directed to efficient solution algorithms for the aforementioned nonconvex finite-sum
or multi-block problems.
Let us start with reviewing a few complexity results associated with existing first-order methods for solving 
the finite-sum problem \eqref{problem_non}.
For simplicity, let us assume that $X = \R^n$ for now.
It is well-known (see, e.g., \cite{Nest04}) that the simple gradient descent (GD) method applied to problem~\eqref{problem_non} requires
${\cal O}(L/\epsilon)$ iterations to find an $\epsilon$-stationary solution, i.e., a point $\bar x$ s.t. $\|\nabla f(\bar x)\|^2 \le \epsilon$.
Since each GD iteration 
requires a full gradient computation, i.e., $m$ gradient computations for $f_i$'s, totally this algorithm needs ${\cal O}(m L / \epsilon)$ gradient computations 
for all the component functions $f_i$'s. 
Ghadimi and Lan~\cite{ghadimi2013stochastic} (see also ~\cite{GhadimiLanZhang2016}) show that by using the stochastic gradient descent (SGD) method, one only needs to compute
the gradient of one randomly selected component function at each iteration, resulting in totally ${\cal O}(L\sigma^2 /\epsilon^2)$ gradient computations
to find a stochastic $\epsilon$-stationary solution of \eqref{problem_non}, i.e., a point $\bar x$ s.t. $\Eb[\|\nabla f(\bar x)\|^2] \le \epsilon$.
Here the expectation is taken w.r.t. some random variables used in the algorithm and $\sigma^2$ denotes their variance. 
Although this complexity bound does not depend on $m$, it 
has a much worse dependence on $\epsilon$ than the GD method for solving problem~\eqref{problem_non}.
Inspired by
the variance reduction techniques originated in convex optimization \cite{johnson2013accelerating}, 
Reddi et al. \cite{reddi2016stochastic,reddi2016fast}, and Allen-Zhu and Hazan \cite{allen2016variance} recently show 
that one only needs $\O\left(m^{2/3}L/\e\right)$ gradient evaluations to find an $\epsilon$-stationary point of \eqref{problem_non},
which significantly improves the bound in \cite{ghadimi2013stochastic} in terms of the dependence on $\epsilon$ and also
dominates the one for GD by a factor of $m^{1/3}$.
However, it remains unknown whether one can further improve this bound in terms of its dependence on $m$
for nonconvex finite-sum optimization especially when $L/\mu$ is large. 

A different line of research aims to incorporate Nesterov's acceleration (momentum) \cite{nesterov1983method} 
into nonconvex optimization. Ghaidmi and Lan \cite{ghadimi2016accelerated}~first established the convergence of
the accelerated gradient method for nonconvex optimization and show that it can improve the complexity of GD if the problem has a large condition number (i.e., $L/\mu$ is large).
Their results were further improved in~\cite{ghadimi2015generalized}, \cite{carmon2016accelerated}, \cite{paquette2017catalyst} and \cite{kong2018complexity}. Currently the best
complexity result, in terms
of total gradient computations, for these methods is given by  $\O\left(m\sqrt{L\mu}/\e\right)$
for unconstrained problems~\cite{kong2018complexity} . However,
it remains unknown if the complexity of such accelerated algorithms can
be further improved in terms of the dependence on $m$, especially when one needs to maintain the  ${\cal O}(1/\epsilon)$ complexity bound
on gradient computations. Note that some nonconvex stochastic accelerated gradient methods have been discussed in \cite{ghadimi2016accelerated}
but they all exhibit a worse ${\cal O}(L\sigma^2/\epsilon^2)$ complexity bounds.


While stochastic and randomized methods are being intensively explored for solving problem~\eqref{problem_non},
most existing studies for the nonconvex multi-block problem in \eqref{c:problem_non} 
have been mainly focused on deterministic batch methods only. Many of these studies aim at the   
generalization of the alternating direction method of multipliers (ADMM) method for nonconvex optimization. For example,
In \cite{hong2016decomposing}, Hong et al. established the complexity for a variant of ADMM for nonconvex multi-block problems, 
see also \cite{hong2016convergence} and \cite{wang2015global} for some previous work on the asymptotic analysis of ADMM for nonconvex optimization.  
In \cite{melo2017iteration-linearized}, Melo and Monteiro presented a linearized proximal multiblock ADMM 
with complexity $\O\left(1/\e\right)$ to attain a nearly feasible $\e$-stationary solution, but all the blocks have to be updated in each iteration.
Later, they proposed a Jacobi-type ADMM in \cite{melo2017iteration-Jacobi} with similar complexity bound, which shows benefits if parallel computing is available. 
While the idea of randomly selecting blocks in nonconvex ADMM has been explored recently, 
these studies focus on the asymptotical convergence of these schemes (e.g., in \cite{hong2016convergence,Ye2018}).
To the best of our knowledge, there does not exist any complexity analysis regarding randomized methods
for solving the nonconvex multi-block problem in \eqref{c:problem_non} in the literature and as a consequence, it remains unclear whether
stochastic or randomized methods are more advantageous over batch ones or not.

Our contribution in this paper mainly exists in the following several aspects.
Firstly, we develop a new randomized algorithm, namely the randomized accelerated proximal gradient (RapGrad) method
for solving problem~\eqref{problem_non} and show that it can significantly improve the
complexity of existing algorithms especially for problems with a large condition number. 
More specifically, we show that RapGrad requires totally
$\O( \mu(m+\sqrt{m L/\mu}) /\e)$ gradient computations in order to find a stochastic $\e$-stationary point.
For  problems with $L/\mu \ge m$, this bound reduces to $\O( \sqrt{mL\mu}/\e)$,
which dominates the best-known batch accelerated gradient methods by a factor of $\sqrt{m}$~\cite{kong2018complexity},
and outperforms those variance-reduced stochastic algorithms~\cite{reddi2016stochastic,reddi2016fast,allen2016variance} by a 
factor of $m^\frac{1}{6} L^\frac{1}{2} / \mu^\frac{1}{2}$ (at least $m^\frac{2}{3}$). In fact, our complexity bound will be better than
the latter algorithms as long as $L/\mu \log (L/\mu)  >  m^{\frac13}$.
Therefore, we provide some affirmative answers regarding whether the complexity bounds of  variance reduced algorithms
and accelerated gradient methods for nonconvex optimization
can be further improved, especially in terms of their dependence on $m$.
To the best of our knowledge, all these complexity results seem to be new in the literature for nonconvex finite-sum optimization.
It is worth noting that some improvement over variance-reduced stochastic algorithms under the region $m \ge L/\mu$ (i.e.,
$L/\mu$ is small)
has been presented recently in \cite{2017arXiv170200763A}.
RapGrad is  a proximal-point type method which iteratively transforms the original nonconvex problem into  
a series of convex subproblems.
In RapGrad, we incorporate a modified optimal randomized incremental gradient method, namely 
the randomized primal-dual gradient (see \cite{lan2017optimal}) to solve these convex subproblems, and
as a consequence, each iteration of RapGrad requires gradient computation for only one randomly selected component function.
In comparison with existing nonconvex proximal-point type methods, the design and analysis of RapGrad 
appear to be more complicated. In particular, RapGrad does not require the computation of full gradients throughout its
entire procedure by properly initializing a few intertwined search points and gradients using information obtained from the previous subproblems. 
This comes with the price of requiring additional storage (memory) for maintaining $\O(m)$ variables (e.g., $\ux^t$).
Moreover, the analysis of RapGrad requires
us to show the convergence for some auxiliary sequences where the gradients are
computed, which has not been established for the original randomized primal-dual gradient method.

Secondly, inspired by RapGrad, we develop a new randomized proximal-point type method, namely
the randomized accelerated proximal dual (RapDual) method, for solving
the nonconvex multi-block problem in \eqref{c:problem_non}. Similarly to RapGrad, this method
solves a series of strongly convex subproblems iteratively generated by adding strongly convex terms, 
via a novel randomized dual
method developed in this paper for solving linearly constrained problems. Each iteration of RapDual requires access to only one randomly selected block,  and
the solution of a relatively easy primal block updating operator.
Note that in order to
guarantee the strong concavity of the Lagrangian dual of the subproblem, we need to assume that $X_m = \R^n$ and the last block $A_m$ is
invertible. Moreover, we assume that it is relatively easy to compute $A_m^{-1}$ (e.g., $A_m$ is the identity matrix, sparse or symmetric diagonally dominant)
to simplify the statement and analysis of the algorithm (see Remark~\ref{remarkinverse} for more discussions).
Let us consider for now the case when
$X_i\equiv \R^{d_i}$ and $A_m = I$. We can show that RapDual can find a solution $(\bx_1,\ldots,\bx_m)$
s.t. $\exists \lambda \in \R^n$, $\Eb[\sum_{i=1}^m\|\nabla f(\bar x_i)+ A_i^{\top}\lbd\|^2] \le \epsilon$ and $\Eb[\|\sum_{i=1}^m A_i \bar x_i - b\|^2] \le \s$
in at most 
\[N(\e,\s):=\tsy{\O\left( m\bar A\sqrt{L\mu}\log\left(\tfrac{L}{\mu}\right)\cdot \max\left\{\tfrac{1}{\e}, ~\tfrac{\|A\|^2}{\s L^2}\right\}\mathcal{D}^0\right)}\]
primal block updates, where  $\bar{A} = \max_{i\in[m-1]} \|A_{i}\|$, $\|\bfA\|^2 = \sum_{i=1}^{m-1} \| A_{i}\|^2$, and $\mathcal{D}^0:=\sum_{i=1}^{m} [f_i(\bx_i^{0})-f_i(x_i^{*})]$. 
Moreover, we demonstrate that the total number primal block updates that RapGrad requires 
can be much smaller, up to a factor of ${\cal O}(\sqrt{m})$, than its batch counterpart.
To the best of our knowledge, this is the first time that the complexity of randomized methods for solving this special class of nonconvex multi-block optimization
 has been established and their possible advantages over batch methods are quantified in the literature.

Thirdly, we perform some numerical experiments on both RapGrad and RapDual for solving nonconvex
finite-sum and multi-block problems in \eqref{problem_non} and \eqref{c:problem_non}
and demonstrate their potential advantages over some existing algorithms.

This paper is organized as follows. In Section \ref{Sec 2}, we present our algorithm RapGrad, and its convergence properties for solvinwg 
the nonconvex finite-sum  problem in \eqref{problem_non}.  RapDual for nonconvex finite-sum  optimization with linear constraints \eqref{c:problem_non} and its  
 convergence analysis are included in  Section \ref{Sec 3}. Section \ref{sec-num} is devoted to some  numerical experiments of our algorithms 
 for the above two types of problems. Finally some concluding remarks are made in Section~\ref{conclusion}.

\subsection{Notation and terminology}Let $\mathbb{R}$ denote the set of real numbers.
All vectors are viewed as column vectors, and
for a vector $x \in \mathbb{R}^d$, we use $x^{\top}$ to denote its transpose.
For any $n \ge 1$, the set of integers $\{1,\ldots,n\}$ is denoted by $[n]$.
We use $\Eb_s[X]$ to denote the expectation of a random variable $X$ on $i_1,\ldots, i_s$.
For a given strongly convex function $\o$, we define the prox-function associated with $\o$ as
\[V_{\o}(x,y): = \o(x)-\o(y)-\la \o^{\prime}(y),x-y\ra, \quad \forall x,y \in X.\]
where $\o^{\prime}(y)\in \partial \o(y)$  is an arbitrary subgradient of $\o$ at $y$.  For any $s\in \R$,  $\lceil s \rceil$ denotes
the nearest integer to $s$ from above.

\section{Nonconvex finite-sum optimization}\label{Sec 2}
In this section, we develop a randomized accelerated proximal gradient (RapGrad) method
for solving the nonconvex finite-sum optimization problem in \eqref{problem_non}
and demonstrate that it can significantly improve the existing rates of convergence
for solving these problems, especially when their objective functions have large condition numbers.
We will describe this algorithm and establish its convergence in Subsections~\ref{sec_SGD_alg}
and \ref{sec_SGD_analysis}, respectively.

\subsection{The Algorithm} \label{sec_SGD_alg}
The basic idea of RapGrad is to solve problem~\eqref{problem_non} iteratively by using the proximal-point type method. 
More specifically, given a current search point  $\bx^{\l-1}$ at the $l$-th iteration, we will employ a randomized accelerated gradient (RaGrad) obtained by properly modifying the
randomized primal-dual gradient method in \cite{lan2017optimal}, to approximately solve
\beq \label{subprob_rand}
\min_{x\in X} ~ \sm f_i(x)+ \frac{3\mu}{2} \|x - \bx^{\l-1}\|^2
\eeq
to compute a new search point $\bx^\l$. Similar idea can be found in \cite{lin2015universal}, where proximal-point method and randomized method are combined to minimize a convex objective.

The algorithmic schemes for RapGrad and RaGrad are described in Algorithm~\ref{alg_non} and Algorithm~\ref{alg:rand}, respectively.
While it seems that we can directly apply the randomized primal-dual gradient method in \cite{lan2017optimal} (or other
fast randomized incremental gradient method) to solve \eqref{subprob_rand}
since it is strongly convex due to  \eqref{assumptions_non2}, a direct application of these methods
would require us to compute the full gradient from time to time whenever a new subproblem needs to be solved.
In fact, if one applies a variance reduced incremental gradient method to solve \eqref{subprob_rand},
the algorithmic scheme would involve three loops (or epochs) and each epoch requires a full gradient computation. 
Moreover, a direct application of these existing first-order methods to solve \eqref{subprob_rand}
would result in some extra logarithmic factor ($\log (1/\epsilon)$) in the final complexity bound as shown in~\cite{carmon2016accelerated}.
Therefore, we 
employed the RaGrad method to solve \eqref{subprob_rand}, which differs from the original randomized primal-dual gradient method
in the following several aspects. Firstly, 
different from the randomized primal-dual gradient method, the design and
analysis of RaGrad does not involve the conjugate functions of $f_i$'s, but only first-order information (function values and gradients). 
Such an analysis
enables us to build a relation between successive search points $\bx^{\l}$, as well as the
convergence of the sequences $\bar\ux^{\l}_i$ where the gradients $\by^{\l}_i$ are computed. 
With these relations at hand, we can determine the number of iterations $s$ required by Algorithm~\ref{alg:rand} to ensure
the overall RapGrad Algorithm to achieve an accelerated rate of convergence.

Second, the original randomized primal-dual gradient method 
requires the computation of  only one randomly
selected gradient at each iteration, and does not require the computation of full gradients from time to time. 
However, it is unclear whether a full pass of all component functions is required 
whenever we solve a new proximal subproblem (i.e., $\bx^{\l-1}$ changes at each iteration). 
It turns out that by properly initializing a few intertwined primal and gradient sequences in RaGrad
using information obtained from previous subproblems,  we will compute full gradient only once for the very first time when 
this method is called, and do not need to compute full gradients any more when solving all other subproblems 
throughout the RapGrad method. 
Indeed, the output $y^i_s$ of RaGrad (Algorithm~\ref{alg:rand})  represent the gradients of $\psi_i$ at the search points $\underline x_i^s$. By using the strong convexity of the objective functions,
we will be able to show that all the search points $\underline x^i_s$, $i = 1, \ldots, m$, will converge,
similarly to the search point $x^s$,  to the optimal solution of the subproblem in \eqref{problem_convex} (see Lemma~\ref{lemma_dist} below).
Therefore, we can use $y^i_s$ to approximate $\nabla \psi_i(x^s)$ and thus remove the necessity of computing the full gradient of $x^s$ when solving the next subproblem.

\begin{algorithm}[H]  \caption{RapGrad for nonconvex finite-sum  optimization}  
\label{alg_non}
\begin{algorithmic} 
\STATE{Let $\bx^0\in X$, and set $\bar\ux^0_i =\bx^0$, $\by^0_i = \n f_i(\bx^0)$, $i=1,\ldots,m$.}
\FOR{$\ell = 1, \ldots, k$}
\STATE{Set $x^{-1}=x^0 = \bx^{\l-1}$, $\ux_i^{0}=\bar\ux^{\l-1}_i$, and $y_i^{0}=\by^{\l-1}_i$},  $i=1,\ldots,m$.
\STATE{Run RaGrad (c.f., Algorithm \ref{alg:rand}) with input $x^{-1}$, $x^0$, $\ux_i^{0}$, $y_i^{0}$, $i=1,\ldots,m$, and $s$  to solve the following subproblem}
\beq \label{problem_convex}
\min_{x\in X} ~ \sm \psi_i(x)+\varphi(x)
\eeq
to obtain output $x^s$, $\underline{x}_i^s$, $y_i^s$ ,  $i=1,\ldots,m$,
where $\psi_i(x) \equiv \psi_i^{\l}(x):= f_i(x)+ \mu\|x-\bx^{\l-1}\|^2$, $i=1,\ldots,m$, and $\varphi(x) \equiv \varphi^{\l}(x):= \tfrac{\mu}{2}\|x-\bx^{\l-1}\|^2$.
\STATE{Set $\bx^{\l}= x^s$, $\bar\ux^{\l}_i=\ux_i^s$ and $\by^{\l}_i=y_i^{s}+2\mu(\bx^{\l-1}-\bx^{\l})$, $i=1,\ldots,m$ (note $\by^{\l}_i=\n \psi_i^{\l+1}(\bar\ux^{\l}_i)$  always holds).}
\ENDFOR
\RETURN $\bx^{\hat\l}$ for some random $\hat \l\in [k]$.
\end{algorithmic}
\end{algorithm}

\begin{algorithm}[H]
\caption{RaGrad for iteratively solving subproblem \eqref{problem_convex}}\label{alg:rand}
\begin{algorithmic}
\STATE{Input $x^{-1}=x^0\in X$,  $\ux_i^0\in X$, $y_i^0$, $i=1,\ldots,m$, number of iterations $s$.
Assume nonnegative parameters $\{\a_t\}$, $\{\tau_t\}$, $\{\eta_t\}$ are given.} 
\FOR{$t = 1, \ldots, s$}
\STATE{1. Generate a random variable $i_t$ uniformly distributed over $[m]$.}
\STATE{2. Update $x^t$ and $y^t$ according to}
\begin{align}
\tilde{x}^t =&~ \a_t (x^{t-1}-x^{t-2}) + x^{t-1}.\label{def_tx}\\
\ux_i^t = &~
\left\{
\begin{array}{ll}
(1+\tau_t)^{-1}(\tilde{x}^t+\tau_t\ux_i^{t-1}), & i = i_t, \\
\ux_i^{t-1}, & i \neq i_t,
\end{array}
\right.\label{def_tux}\\
y_i^t =&~
\left\{
\begin{array}{ll}
\n \psi_i(\ux_i^t), & i = i_t, \\
y_i^{t-1}, & i \neq i_t,
\end{array}
\right.\label{def_ty}\\
\tilde{y_i}^t =&~m(y_i^t-y_i^{t-1})+y_i^{t-1}, \quad \forall i=1,\ldots,m\\
x^t = &~
\argmin_{x\in X}~  \varphi(x) + \lla \sm \tilde y_i^t,x\rra + {\eta_t}V_{\varphi}(x,x^{t-1}).\label{def_xt}
\end{align}
\ENDFOR
\RETURN $x^s$, $\ux_i^s$, and $y_i^s$, $i=1,\ldots,m$.
\end{algorithmic}
\end{algorithm}

Before establishing the convergence of the RapGrad method, we first need to define
an approximate stationary point for problem~\eqref{problem_non}. 
A point $x \in X$ is called an approximate stationary point if it
sits within a small neighborhood of a point $\hat x \in X$ which approximately satisfies the first-order optimality condition.

\begin{definition} \label{def_solution_acc}
A point $x \in X$ is called an  $(\e,\d)$-solution of \eqref{problem_non} if there exists some $\hat x\in X$ such that
\[[d\left( \n f(\hat x),-N_X(\hat x)\right)]^2\leq \e \quad \text{and}\quad \|x-\hat x\|^2\leq \d.\] 
A stochastic $(\e,\d)$-solution of \eqref{problem_non} is one such that 
\[\Eb[d\left( \n f(\hat x),-N_X(\hat x)\right)]^2\leq \e \quad \text{and}\quad \Eb\|x-\hat x\|^2\leq \d.\] 
Here, $d(x,Z):=\inf_{z\in Z} \|x-z\|$ denotes the distance from $x$ to set $Z$, and
$N_{X} (\hat x):=\{x\in\R^n| ~\la x, y- \hat x\ra\leq 0  \text{~for all~}y \in X\}$ denotes the normal cone of $X$ at $\hat x$.  
\end{definition}

To have a better understanding of the above definition,  let us consider the unconstrained problem \eqref{problem_non}, i.e., $X=\R^n$. 
Suppose that $x \in X$ is an $(\e,\d)$-solution with $\d=\e/L^2$. Then
there exists $\hat x \in X$ s.t.
$\|\n f(\hat x)\|^2\leq \e$ and  $\| x-\hat x\|^2\leq \e/L^2$, which implies that
\begin{align}
\|\n f(x)\|^2&=\|\n f(x)-\n f(\hat x)+\n f(\hat x)\|^2
\leq  2 \|\n f(x)-\n f(\hat x)\|^2+2\|\n f(\hat x)\|^2\nn\\
&\leq  2 L^2\|x-\hat x\|^2+2\|\n f(\hat x)\|^2
\leq 4\e.\label{quality}
\end{align}
Moreover, if $X$ is a compact set and $x \in X$ is an $(\e,\d)$-solution, we can bound strong gap as follows:

\beq \label{quality1}
\begin{aligned}
{\rm gap}(x) 
&:= \max_{z \in X} \langle \n f(x), x - z\rangle \\ 
&= \max_{z \in X} \langle \n f(x)- \n f(\hat x), x - z\rangle 
+\max_{z \in X} \langle \n f(\hat x) , x- \hat x \rangle
+\max_{z \in X} \langle \n f(\hat x) , \hat x- z\rangle \\
&\le L\sqrt{\d} D_X + \sqrt{\d} \|\n f(\hat x)\| +  \sqrt{\e} D_X,
 \end{aligned}
\eeq
where $D_X := \max_{x_1, x_2 \in X} \|x_1 - x_2\|$. In comparison with the two well-known criterions in \eqref{quality} and \eqref{quality1}, 
the criterion given in Definition~\ref{def_solution_acc}
seems to be applicable to a wider class of problems and is particularly
suitable for proximal-point type methods (see \cite{davis2017proximally} for a related notion).

We are now ready to state the main convergence properties for RapGrad.

\begin{theorem}\label{Thm_main}
 Let the iterates $\bx^{\l}$, $\l= 1, \ldots, k$, be generated by Algorithm \ref{alg_non}
 and $\hat \l$ be randomly selected from $[k]$. 
Suppose that in Algorithm \ref{alg:rand}, the number of iterations $s=\lceil-\log \wt M / \log \a\rceil$ with
\beq\label{parameters1}
\wt M:=\tsy{6\left(5+\tfrac{2L}{\mu}\right)}\max\left\{\tfrac{6}{5},\tfrac{L^2}{\mu^2}\right\}, \quad \a=1-\tfrac{2}{m\left(\sqrt{1+16c/m}+1\right)}, \quad c = 2+\tfrac{L}{\mu},
\eeq
and other parameters are set to 
\beq\label{parameters2}
\a_t =\a,\quad \g_t = \a^{-t},  
\quad \tau_t =  \tfrac{1}{m(1-\a)}-1,\quad \mbox{and} \quad \eta_t = \tfrac{\a}{1-\a},\quad \forall t = 1, \ldots, s.
\eeq
Then we have
\begin{align*}
\Eb\left[d\left( \n f(x^{\hat\l}_*),-N_X(x_*^{\hat\l})\right)\right]^2
&\leq \tfrac{36\mu}{k}[f(\bx^{0})-f(x^{*})],\\
\Eb\|\bar x^{\hat\l}-x^{\hat\l}_*\|^2&\leq \tfrac{4\mu}{kL^2}[f(\bx^{0})-f(x^{*})],
\end{align*}
where $x^*$ and $x^{\l}_*$ denote the optimal solutions to problem \eqref{problem_non} and the $\l$-th subproblem \eqref{subprob_rand}, respectively.
\end{theorem}

Theorem \ref{Thm_main} guarantees, in expectation, the existence of an approximate stationary point $x_*^{\hat\l}$, which is the optimal solution to the $\hat\l$-th subproblem. 
Though $x_*^{\hat\l}$ is unknown to us, we can output the computable solution $\bar x^{\hat\l}$ since it is close enough to $x_*^{\hat\l}$. Moreover,
its quality can be directly measured by \eqref{quality} and \eqref{quality1} under certain important circumstances.

In view of Theorem \ref{Thm_main}, we can bound the total number of gradient evaluations required by RapGrad to yield a stochastic $(\e,\d)$-solution of \eqref{problem_non}. 
Indeed, observe that the full gradient is computed only once  in the first outer loop,
and that for  each subproblem \eqref{problem_non}, 
we only need to compute $s$ gradients with
\[\tsy{s= \left\lceil-\tfrac{\log \wt M}{\log {\a}}\right\rceil\sim\O\left(\left(m+\sqrt{m\tfrac{ L}{\mu}}\right)\log \left(\tfrac{ L}{\mu}\right)\right)}.\]
Hence, the total number of gradient evaluations performed by RapGrad can be bounded by

\[
N(\e,\d):=\tsy{\O\left(m + \left (m + \sqrt{\tfrac{mL}{\mu}} \right) \log\tfrac{L}{\mu} +
\mu\left(m+\sqrt{m\tfrac{L}{\mu}}\right)\log \left(\tfrac{ L}{\mu}\right)\cdot\max\left\{\tfrac{1}{\d L^2},~\tfrac{1}{\e}\right\}D^0\right)},
\]where  $D^0:=f(\bx^{0})-f(x^{*})$.  As a comparison,
the batch version of this algorithm, obtained by viewing $\sm f_i(x)$ as a single component, would update 
all the $\ux_i^t$ and $y_i^t$ for $i=1,\ldots,m$, in \eqref{def_tux} and \eqref{def_ty} at each iteration, 
and hence would require 

\[\hat N(\e,\d):=\tsy{\O\left(m\sqrt{\tfrac{L}{\mu}} \log\tfrac{L}{\mu}   +  m\sqrt{L\mu}\log \left(\tfrac{L}{\mu}\right)\cdot \max\left\{\tfrac{1}{\d L^2},~\tfrac{1}{\e}\right\}D^0\right)}\]
gradient evaluations to compute an $(\e,\d)$-solution of \eqref{problem_non}. 


It is worth noting that both complexity bounds $N(\e,\d)$ and $\hat N(\e,\d)$
will go to $+\infty$ as $\mu$ tends to $0$. This indicates that 
the proximal-point method
should not be applied to the situation when $\mu$ is too small, no matter whether the randomized or batch version is used.
However, if $\mu$ is indeed small, say $\mu < \epsilon$,
we can see that the assumption in \eqref{assumptions_non2} also holds for $\mu = \epsilon$.
Therefore, we can assume that $\mu \ge \epsilon$ when applying RapGrad.
A closer examination to this issue reveals that
the algorithm (i.e., RaGrad) that we used to solve the strongly convex subproblems
is not optimal if the strongly convex modulus $\mu$ is too small.
Hence, a different  remedy would be to apply an optimal randomized
incremental gradient method which can yield the best possible complexity bound
even if $\mu$ is small. Using this latter approach, we possibly do not need to fix $\epsilon$ a priori as in the former one.  
However,  this type of optimal
randomized incremental gradient method has not been developed  until
 recently (see \cite{LanLiZhou19-1}), about one year later after our paper was initially released.

For problems with $L/\mu \ge m$,  RapGrad can potentially save the total number of
gradient computations up to a factor of ${\cal O}(\sqrt{m})$ gradient evaluations than its batch counterpart as well as
other deterministic batch methods reported in \cite{paquette2017catalyst,kong2018complexity}. 
It is also interesting to compare RapGrad with those variance-reduced stochastic algorithms~\cite{reddi2016stochastic,reddi2016fast,allen2016variance}.
For simplicity, consider for the case when $\d=\e/L^2$ and $X \equiv \R^n$. In this case, the complexity bound of RapGrad, given by $\O( \sqrt{mL\mu}/\e)$,
is smaller than those of variance-reduced stochastic algorithms~\cite{reddi2016stochastic,reddi2016fast,allen2016variance} by a 
factor of ${\cal O}(m^\frac{1}{6} L^\frac{1}{2} / \mu^\frac{1}{2})$, which must be greater that
${\cal O}(m^\frac{2}{3})$ due to $L/\mu \ge m$. In fact, our complexity bound minorizes
those for variance-reduced stochastic algorithms as long as $L/\mu \log (L/\mu)  >  m^{\frac13}$.

Theorem~\ref{Thm_main} only shows the convergence of RapGrad in expectation. 
Similarly to the nonconvex SGD methods in \cite{ghadimi2013stochastic,GhadimiLanZhang2016}, we can 
establish and then further improve the convergence of RapGrad with overwhelming probability  by using a two-phase procedure, where
one computes a short list of candidate solutions in the optimization phase by  either taking a few independent runs of RapGrad 
or randomly selecting a few solutions from the trajectory of RapGrad, and then chooses the best solution, e.g.,
in terms of either \eqref{quality} and \eqref{quality1}, in the post-optimization phase.
  
\subsection{Convergence analysis for RapGrad}\label{sec_SGD_analysis}
In this section, we will first develop the convergence results for Algorithm \ref{alg:rand} applied to the convex finite-sum subproblem~\eqref{problem_convex},
and then using them to establish the convergence of RapGrad.
Observe that the component functions $\psi_i$ and $\varphi$ in ~\eqref{problem_convex} satisfy:
\begin{enumerate}[label=(\Alph*)]
\item
$\tfrac{\mu}{2}\|x-y\|^2\leq \psi_i(x)-\psi_i(y)-\la\n \psi_i(y),x-y\ra \leq   \tfrac{\hat L}{2}\|x-y\|^2,\ \ \forall x, y\in X, \quad i=1,\ldots,m,$
\item
$\varphi(x)-\varphi(y)-\la\n \varphi(y),x-y\ra \geq \tfrac{\mu}{2}\|x-y\|^2, \ \ \forall x, y\in X,$
\end{enumerate}
where $\hat L = L+2\mu$.

We first state some simple relations about the iterations generated by Algorithm \ref{alg:rand}.
\begin{lemma}\label{lemma_full}
Let $\hat{\ux}_i^t = (1+\tau_t)^{-1}(\tilde{x}^t+\tau_t\ux_i^{t-1})$, for $i = 1,\ldots, m,~ t=1,\ldots, s$. 
\begin{align}
\Eb_{i_t}  [\psi(\hat{\ux}_i^t)] &= m \psi(\ux_i^t)-(m-1) \psi(\ux_i^{t-1}) \label{eqn:full1},\\
\Eb_{i_t} [\n \psi(\hat{\ux}_i^t)] &= m\n \psi(\ux_i^t)-(m-1)\n \psi(\ux_i^{t-1})=\Eb_{i_t} [\tilde y_i^t]\label{eqn:full2}.
\end{align}
\begin{proof} 
By the definition of  $\hat{\ux}_i^t $, it is easy to see that $\Eb_{i_t} [\ux_i^t]= \tfrac{1}{m}\hat{\ux}_i^t+\tfrac{m-1}{m}\ux_i^{t-1}$, 
thus $\Eb_{i_t} [\psi_i(\ux_i^t)]= \tfrac{1}{m}\psi_i(\hat{\ux}_i^t)+\tfrac{m-1}{m}\psi_i(\ux_i^{t-1})$, 
and $\Eb_{i_t} [\n \psi_i(\ux_i^t)]= \tfrac{1}{m} \n \psi_i(\hat{\ux}_i^t)+\tfrac{m-1}{m}\n \psi_i(\ux_i^{t-1})$, 
which combined with the fact $\tilde{y}_i^t =~m(y_i^t-y_i^{t-1})+y_i^{t-1}$, gives us the desired relations.
\end{proof}
\end{lemma}

Lemma~\ref{lemma_dist} below describes an important result about Algorithm \ref{alg:rand}, which
improves Lemma 7 of \cite{lan2017optimal} by showing the convergence of ${\ux}^{s}_i$.
The proof of this result is more involved and will be deferred  in Appendix~\ref{pf_lemma2}.
\begin{lemma}\label{lemma_dist}
Let the iterates $x^t$ and $y^t$, for $t = 1, \ldots, s$, be generated by Algorithm \ref{alg:rand}
and $x^*$ be an optimal solution of \eqref{problem_convex}.
If the parameters in Algorithm \ref{alg:rand} satisfy for all $t = 1, \ldots, s-1$,
\begin{align}
\a_{t+1}\g_{t+1}= &~ \g_{t}, \label{eqn:ss1}\\
\g_{t+1}[m(1+\tau_{t+1})-1] \leq &~ m\g_{t}(1+\tau_{t}),\label{eqn:ss2}\\
\g_{t+1}\eta_{t+1}\leq &~ \g_{t}(1+\eta_{t}),\label{eqn:ss3}\\
 \tfrac{\eta_s \mu}{4}\geq &~ \tfrac{(m-1)^2\hat L}{m^2\tau_s},\label{eqn:ss4}\\
\tfrac{\eta_{t}\mu}{2}\geq &~\tfrac{\a_{t+1}\hat L}{\tau_{t+1}}+\tfrac{(m-1)^2\hat L}{m^2\tau_{t}}  ,\label{eqn:ss5}\\
\tfrac{\eta_s\mu}{4} \geq &~\tfrac{\hat L}{m(1+\tau_s)}, \label{eqn:ss6}
\end{align}
then we have
\begin{align*}
\Eb_s\left[\g_s(1+\eta_s)V_\varphi(x^*,x^s)+\smn\tfrac{\mu\g_s (1+\tau_s)}{4}\|{\ux}^{s}_i-x^*\|^2\right] &\leq 
\g_1\eta_1  \Eb_sV_\varphi(x^*,x^0)
+\smn\tfrac{\g_1[(1+\tau_1)-1/m]\hat L}{2}\Eb_s\|{\ux}^{0}_i-x^*\|^2.
\end{align*}
\end{lemma}

With the help of Lemma \ref{lemma_dist},
we now establish the main convergence properties of Algorithm \ref{alg:rand}.

\begin{theorem}\label{Thm_convex}
Let $x^*$ be an optimal solution of  \eqref{problem_convex}, and suppose that the parameters $\{\a_t\}$, $\{\tau_t\}$, $\{\eta_t\}$ and $\{\g_t\}$ are set as in \eqref{parameters1} and  \eqref{parameters2}. 
If $\varphi(x)=\tfrac{\mu}{2}\|x-z\|^2$, for some $z\in X$, then, for any $s \ge 1$, we have
\begin{align*}
\Eb_s\left[\|x^*-x^s\|^2\right]&~\leq \tsy{\a^s(1+2\tfrac{\hat L}{\mu})~\Eb_s\left[\|x^*-x^0\|^2+\sm\|{\ux}^{0}_i-x^0\|^2\right]},\\
\Eb_s\left[\sm\|{\ux}^{s}_i-x^s\|^2\right]&~\leq\tsy{6\a^s(1+2\tfrac{\hat L}{\mu})~\Eb_s\left[\|x^*-x^0\|^2+\sm\|{\ux}^{0}_i-x^0\|^2\right]}.
\end{align*}
\end{theorem}
\begin{proof}
It is easy to check that \eqref{parameters1} and  \eqref{parameters2} satisfy conditions \eqref{eqn:ss1}, \eqref{eqn:ss2}, \eqref{eqn:ss3} \eqref{eqn:ss4}, \eqref{eqn:ss5}, and \eqref{eqn:ss6}.
Then by Lemma \ref{lemma_dist}, we have 
\begin{align}\label{thepx}
\Eb_s\left[V_\varphi(x^*,x^s)+\smn\tfrac{\mu}{4m}\|{\ux}^{s}_i-x^*\|^2\right]\leq \a^s\Eb_s\left[V_\varphi(x^*,x^0)+\smn\tfrac{\hat L}{2m}\|{\ux}^{0}_i-x^*\|^2\right].
\end{align}
Since $\varphi(x)=\tfrac{\mu}{2}\|x-z\|^2$, we have $V_\varphi(x^*,x^s) = \tfrac{\mu}{2}\|x^*-x^s\|^2$, and $V_\varphi(x^0,x^s) =\tfrac{\mu}{2}\|x^*-x^0\|^2$. Plugging into \eqref{thepx}, we obtain the following two relations: 
\begin{align*}
\Eb_s\left[\|x^*-x^s\|^2\right]&~\leq \a^s\Eb_s\left[\|x^*-x^0\|^2+\smn\tfrac{\hat L}{mr}\|{\ux}^{0}_i-x^*\|^2\right]\nn\\
&~\leq\a^s\Eb_s\left[\|x^*-x^0\|^2+\smn\tfrac{\hat L}{mr}(2\|{\ux}^{0}_i-x^0\|^2+2\|x^0-x^*\|^2)\right]\nn\\
&~=\tsy{\a^s~\Eb_s\left[(1+2\tfrac{\hat L}{\mu})\|x^*-x^0\|^2+\smn\tfrac{2\hat L}{m\mu}\|{\ux}^{0}_i-x^0\|^2\right]}\nn\\
&~\leq \tsy{\a^s(1+2\tfrac{\hat L}{\mu})~\Eb_s\left[\|x^*-x^0\|^2+\smn\tfrac{1}{m}\|{\ux}^{0}_i-x^0\|^2\right]},\\
\Eb_s\left[\sm\|{\ux}^{s}_i-x^*\|^2\right]
&~\leq \tsy{2\a^s\Eb_s\left[\|x^*-x^0\|^2+\smn\tfrac{\hat L}{m\mu}\|{\ux}^{0}_i-x^*\|^2\right]}\nn\\
&~\leq \tsy{2\a^s(1+\frac{2\hat L}{\mu})\Eb_s\left[\|x^*-x^0\|^2+\sm\|{\ux}^{0}_i-x^0\|^2\right]}.
\end{align*}
In view of the above two relations, we have
\begin{align*}
\Eb_s\left[\sm\|{\ux}^{s}_i-x^s\|^2\right]
&~\leq \tsy{\Eb_s\left[\sm2(\|{\ux}^{s}_i-x^*\|^2+\|x^*-x^s\|^2)\right]}\nn\\
&~=2\Eb_s\left[\sm\|{\ux}^{s}_i-x^*\|^2\right]+2\Eb_s\|x^*-x^s\|^2\nn\\
&~\leq \tsy{6\a^s(1+2\tfrac{\hat L}{\mu})~\Eb_s\left[\|x^*-x^0\|^2+\sm\|{\ux}^{0}_i-x^0\|^2\right]}.
\end{align*}
\end{proof}

In view of Theorem \ref{Thm_convex}, Algorithm \ref{alg:rand} applied to subproblem~\eqref{problem_convex} exhibits a fast linear rate of convergence. 
Actually, as shown below we do not need to solve the subproblem too accurately, and a constant number of iteration of Algorithm \ref{alg:rand} 
for each subproblem is enough to guarantee the convergence of Algorithm \ref{alg_non}.

\begin{lemma}\label{lemma_final}
Let the number of inner iterations $s\geq \left\lceil-\log (7M/6)/\log \a\right\rceil$ with $M:=\tsy{6(5+2L/\mu)}$ be given.
Also let the iterates $\bx^{\l}$, $\l = 1, \ldots, k$, be generated by Algorithm \ref{alg_non}, 
and $\hat \l$ be randomly selected from $[k]$.
Then 
\begin{align*}
\Eb \|x^{\hat\l}_*-\bx^{\hat\l-1}\|^2&\leq \tfrac{4(1-M\a^s)}{k\mu(6-7M\a^s)}[f(\bx^{0})-f(x^{*})],\\
\Eb\|x^{\hat\l}_*-\bx^{\hat\l}\|^2&\leq \tfrac{2M\a^s}{3k\mu(6-7M\a^s)}[f(\bx^{0})-f(x^{*})],
\end{align*}
where $x^*$ and $x^{\l}_*$ are the optimal solutions to problem \eqref{problem_non} and the $\l$-th subproblem \eqref{subprob_rand}, respectively.
\end{lemma}	
\begin{proof}
According to Theorem \ref{Thm_convex} (with $\hat L = 2\mu+ L$), we have,  for $\l\geq 1$, 
\begin{align}
 \Eb\|x_*^{\l}-\bx^{\l}\|^2 &\leq \tsy{\a^s(5+\tfrac{2L}{\mu})\Eb\left[\|x^{\l}_*-\bx^{\l-1}\|^2+\smn\tfrac{1}{m}\|\bar{\ux}^{\l-1}_i-\bx^{\l-1}\|^2\right]}\nn\\
 &\leq \tsy{\tfrac{M\a^s}{6}~\Eb\left[\|x^{\l}_*-\bx^{\l-1}\|^2+\smn\tfrac{1}{m}\|\bar{\ux}^{\l-1}_i-\bx^{\l-1}\|^2\right]},\label{starkk}\\
\Eb\left[\sm\|\bar{\ux}^{\l}_i-\bx^{\l}\|^2\right]&\leq \tsy{4\a^s(5+\tfrac{2L}{\mu})~\Eb\left[\|x_*^{\l}-\bx^{\l-1}\|^2+\smn\tfrac{1}{m}\|{\bar\ux}^{\l-1}_i-\bx^{\l-1}\|^2\right]}\nn\\
&\leq\tsy{M\a^s~\Eb\left[\|x_*^{\l}-\bx^{\l-1}\|^2+\smn\tfrac{1}{m}\|{\ux}^{\l-1}_i-\bx^{\l-1}\|^2\right]}.\label{hatkk}
\end{align}
By induction on \eqref{hatkk} and noting $\bar{\ux}^{0}_i=\bx^0$, $i=1,\ldots,m$, we have 
\[\Eb\left[\sm\|\bar{\ux}^{\l}_i-\bx^{\l}\|^2\right]\leq \tsy{\sum_{j=1}^{\l}} (M\a^s)^{\l-j+1}\Eb \|x_*^{j}-\bx^{j-1}\|^2.\]
In view of the above relation and \eqref{starkk}, for $\l\geq 2$, we have 
\[\Eb\|x_*^{\l}-\bx^{\l}\|^2 \leq \tsy{\tfrac{M\a^s}{6}~\Eb\left[\|x^{\l}_*-\bx^{\l-1}\|^2+ \sum_{j=1}^{\l-1} (M\a^s)^{\l-j}\|x_*^{j}-\bx^{j-1}\|^2 \right]   }.\]
Summing up both sides of the above inequality from $\l=1$ to $k$, we then obtain
\begin{align}\label{iteration}
\sl\Eb\|x_*^{\l}-\bx^{\l}\|^2 &\leq \tsy{\tfrac{M\a^s}{6}\Eb\left[\|x^{1}_*-\bx^{0}\|^2+\sll\left(\|x^{\l}_*-\bx^{\l-1}\|^2+ \sum_{j=1}^{\l-1} (M\a^s)^{\l-j}\|x_*^{j}-\bx^{j-1}\|^2 \right)\right]    }\nn\\
& = \tsy{\tfrac{M\a^s}{6}~\Eb\left[\|x^k_*-\bx^{k-1}\|^2+\slk\left(\tfrac{1}{1-M\a^s}-\tfrac{(M\a^s)^{k+1-\l}}{1-M\a^s}\right)\|x^{\l}_*-\bx^{\l-1}\|^2\right]    }\nn\\
& \leq \tsy{\tfrac{M\a^s}{6(1-M\a^s)}\sl\Eb\|x^{\l}_*-\bx^{\l-1}\|^2    }.
\end{align}
Using the fact that $x^{\l}_*$ is optimal to the $\l$-th subproblem, and letting $x^{0}_*$ = $\bx^0$ ($x^0_*$ is a free variable), we have
\[\sl [\psi^{\l}(x_*^{\l})+\varphi^{\l}(x_*^{\l})]\leq\sl [\psi^{\l}(x_*^{\l-1})+\varphi^{\l}(x_*^{\l-1})],\]
which, in view of the definition of $\psi^{\l}$ and $\varphi^{\l}$, then implies that 
\beq\label{rand_non1}
\sl \Eb[f(x^{\l}_*)+\tfrac{3\mu}{2} \|x^{\l}_*-\bx^{\l-1}\|^2]
 \le
 \sl \Eb[f(x^{\l-1}_*)+\tfrac{3\mu}{2} \|x_{*}^{\l-1}-\bx^{\l-1}\|^2].
\eeq
Combining \eqref{iteration} and \eqref{rand_non1}, we obtain
\begin{align}\label{rand_non2}
\tsy{\tfrac{3\mu}{2}} \sl \Eb\|x^{\l}_*-\bx^{\l-1}\|^2
&\le\sl \Eb\{f(x^{\l-1}_*)-f(x^{\l}_*)\}+ \tsy{\tfrac{3\mu}{2}}\sl\Eb\|x_{*}^{\l-1}-\bx^{\l-1}\|^2\nn\\
&\le\sl \Eb\{f(x^{\l-1}_*)-f(x^{\l}_*)\}+ \tsy{\tfrac{3\mu}{2}}\sl\Eb\|x_{*}^{\l}-\bx^{\l}\|^2\nn\\
&\le\sl \Eb\{f(x^{\l-1}_*)-f(x^{\l}_*)\}+ \tsy{\tfrac{3\mu}{2}\tfrac{M\a^s}{6(1-M\a^s)}\sl\Eb\|x^{\l}_*-\bx^{\l-1}\|^2 }.
\end{align}
Using \eqref{rand_non2},  \eqref{iteration} and the condition on $s$, we have
\begin{align*}
\sl\Eb\|x^{\l}_*-\bx^{\l-1}\|^2&\leq \tfrac{4(1-M\a^s)}{\mu(6-7M\a^s)}[f(\bx^{0})-f(x^{*})],\\
\sl\Eb\|x^{\l}_*-\bx^{\l}\|^2&\leq \tfrac{2M\a^s}{3\mu(6-7M\a^s)}[f(\bx^{0})-f(x^{*})].
\end{align*}
Our results then immediately follow since
$\hat\l$ is chosen randomly in $[k]$.
\end{proof}

\vgap

Now we are ready to prove Theorem \ref{Thm_main} using all the previous results we have developed.

\noindent {\bf Proof of Theorem \ref{Thm_main}.}
By the optimality condition of the $\hat\l$-th subproblem \eqref{subprob_rand},  
\beq            
\n \psi^{\hat\l}(x_*^{\hat\l})+\n \varphi^{\hat\l}(x_*^{\hat\l})\in -N_X(x_*^{\hat\l}). 
\eeq
From the definition of $\psi^{\hat\l}$ and $\varphi^{\hat\l}$, we have 
\begin{align}
\n f(x_*^{\hat\l})+3\mu(x_*^{\hat\l}-\bx^{\hat{\l}-1})\in -N_X(x_*^{\hat\l}).
\end{align}
From the optimality condition of \eqref{def_xt}, we obtain
\begin{align}\label{opt_x}
\varphi(x^t)-\varphi(x^*)+\lla \sm \tilde y_i^t, x^t-x^*\rra \leq \eta_t V_\varphi(x^*,x^{t-1})-(1+\eta_t)V_\varphi(x^*,x^t)-\eta_t V_\varphi(x^t,x^{t-1}).
\end{align}
Using the above relation and Lemma 
\ref{lemma_final}, we have
\begin{align*}
\Eb\|\bar x^{\hat\l-1}-x^{\hat\l}_*\|^2&\leq \tfrac{4(1-M\a^s)}{k\mu(6-7M\a^s)}[f(\bx^{0})-f(x^{*})] \leq \tfrac{4}{k\mu}[f(\bx^{0})-f(x^{*})],\\
\Eb\left[d\left( \n f(x^{\hat\l}_*),-N_X(x_*^{\hat\l})\right)\right]^2
&\leq \Eb\|3\mu(\bx^{\hat\l-1}-x^{\hat\l}_*)\|^2 \leq \tfrac{36\mu}{k}[f(\bx^{0})-f(x^{*})],\\
\Eb\|\bar x^{\hat\l}-x^{\hat\l}_*\|^2&\leq \tfrac{2M\a^s}{3k\mu(6-7M\a^s)}[f(\bx^{0})-f(x^{*})]
\leq \tfrac{4M\a^s}{k\mu}[f(\bx^{0})-f(x^{*})]\\
&\leq \tfrac{4\mu}{kL^2}[f(\bx^{0})-f(x^{*})].
\end{align*}
\endproof

\section{Nonconvex multi-block optimization with linear constraints}\label{Sec 3}
In this section, we present a randomized accelerated proximal dual (RapDual) algorithm for solving the nonconvex multi-block optimization problem 
in \eqref{c:problem_non} and show the potential advantages in terms of the total number of block updates.  

As mentioned in Section 1, we assume the inverse of the last block of the constraint matrix is easily computable. Hence, 
denoting $\bfA_i = A_m^{-1} A_i$, $i = 1, \ldots, m-1$ and $\bfb = A_m^{-1}$, we can reformulate  
problem \eqref{c:problem_non} as
\begin{align}\label{c:problem_convex} 
\min_{\Bx\in X, ~x_m\in \R^{n}} &~ f(\Bx)+f_m(x_m),\nn\\
\text{s.t. } &~\bfA \Bx +x_m= \bfb,
\end{align}
where $f(\Bx):=\sum_{i=1}^{m-1} f_i(x_i)$,  $X= X_1 \times \ldots \times X_{m-1}$, $\bfA =[\bfA_{1},\ldots, \bfA_{m-1}]$, and ${\textbf{x}} = (x_1,\ldots,x_{m-1})$.
It should be noted that except for some special cases, the computation of $A_m^{-1}$
requires up to ${\cal O}(n^3)$ arithmetic operations, which will be a one-time computational cost added on top of
the overall computational cost of our algorithm  (see Remark~\ref{remarkinverse} below for more discussions). 
%

One may also reformulate  problem \eqref{c:problem_convex} in the form of  \eqref{problem_non} and directly apply
Algorithm \ref{alg_non}  to solve it.  More specifically, substituting $x_m$ with $\bfb-\bfA\Bx$ in the objective function of \eqref{c:problem_convex}, we obtain
\begin{equation} \label{reformulation}
\min_{\Bx\in X} ~ \tsum_{i=1}^{m-1}f_i(B_i\Bx)+f_m(\bfb-\bfA\Bx),
\end{equation}
where $B_i=(\textbf{0},\ldots,I,\ldots, \textbf{0})$ with the $i$-th block given a $d_i\times d_i$ identity matrix and hence $x_i = B_i\Bx$. 
However, this method will be inefficient since we enlarge the dimension of each $f_i$ from $d_i$ to $\sum_{i=1}^{m-1}d_i$ and as a result, every block has to be updated in each iteration.
One may also try to apply a nonconvex randomized block coordinate descent method~\cite{DangLan2015} to solve the above reformulation. However,
such methods do not apply to the case when $f_i$ are both nonconex and nonsmooth.
This motivates us to design the new RapDual method which
requires to update only a single block at a time, applies to the case when $f_i$ is nonsmooth
and achieves an accelerated rate of convergence when $f_i$ is smooth.

\subsection{The Algorithm}\label{sec_cSGD_alg}
The main idea of RapDual is similar to the one used to design the RapGrad method introduced in Section \ref{sec_SGD_alg}.
Given the proximal points $\bBx^{\l-1}$ and $\bx_m^{\l-1}$ from the previous iteration, we define a new
proximal subproblem as 
\begin{align}\label{csubprob}
\min_{\Bx\in X,x_m\in \R^{n}} &~  \psi(\Bx)+ \psi_m(x_m)\nn\\
\text{s.t. } &~\bfA\Bx+x_m = \bfb,
\end{align}
where $\psi(\Bx):=f(\Bx)+ \mu\|\Bx-\bBx^{\l-1}\|^2$ and $\psi_m(x_m):=f_m(x_m)+ \mu\|x_m- x_m^{\l-1}\|^2$.
Obviously,  RaGrad does not apply directly to this type of subproblem. 
In this subsection, 
we present a new randomized algorithm, named the randomized accelerated dual (RaDual) method
to solve the subproblem in \eqref{csubprob}, which will be iteratively called by the RapDual method 
to solve problem~\eqref{c:problem_convex}. 

RaDual (c.f. Algorithm \ref{calg:rand}) can be viewed as a randomized primal-dual type method. Indeed, by the method of multipliers and Fenchel conjugate duality, we have
\begin{align}\label{eqn:saddle}
&\min_{\Bx\in X,x_m\in \R^{n}}   \{ \psi(\Bx)+\psi_m(x_{m}) + \max_{y\in \mathbb{R}^{n}} \lla \tsum_{i=1}^{m}\bfA_ix_i -\bfb , y\rra \}
\nn \\
=&\min_{\Bx\in X}  \{\psi(\Bx) + \max_{y\in \mathbb{R}^{n}} [\lla \bfA\Bx-\bfb , y\rra+ \min_{x_m\in\R^{n}}\left\{\psi_m(x_m)+\la x_m,y\ra\right\}]\}\nn\\
=&\min_{\Bx\in X}  \{ \psi(\Bx) + \max_{y\in \mathbb{R}^{n}} [\lla \bfA\Bx-\bfb , y\rra-h(y) ]\},
\end{align}
where $h(y):=-\min_{x_m\in\R^{ n}} \{\psi_m(x_m)+\la x_m,y \ra\}=\psi_m^*(-y)$.
Observe that the above saddle point problem 
is both strongly convex in $\Bx$  and strongly concave in $y$.  Indeed, $\psi(\Bx)$ is strongly convex due to the added proximal term. Moreover,
since $\psi_m$ has $\hat L$-Lipschitz continuous gradients, $h(y)=\psi_m^*(-y)$ is $1/\hat L$-strongly convex. 
Using the fact that $h$ is strongly convex,
we can see that \eqref{c:algo2-0}-\eqref{c:algo2} in Algorithm~\ref{calg:rand} is equivalent to a dual mirror-descent step with a properly chosen distance generating function $V_{  h}(y,y^{t-1})$.
Specifically,
\begin{align*}
y^t &=\textstyle\argmin\limits_{y \in \R^n}~  h(y)+\la -\bfA\tilde{\Bx}^t+\bfb ,y\ra + \tau_t V_{  h}(y,y^{t-1})\\
&=\textstyle\argmax\limits_{y \in \R^n}~  \la (\bfA\tilde{\Bx}^t-\bfb +\tau_t\nabla   h(y^{t-1}))/(1+\tau_t),y\ra -h(y)\\
&=\nabla   h^*[(\bfA\tilde{\Bx}^t-\bfb +\tau_t\nabla   h(y^{t-1}))/(1+\tau_t)].
\end{align*}
If we set $g^0=\nabla   h(y^0)=-x_m^0$, then it is easy to see by induction that
$g^t = (\tau_t g^{t-1}+\bfA\tilde{\Bx}^t-\bfb )/(1+\tau_t)$, 
and $y^t = \nabla   h^*(g^t)$ for all $t\geq 1$. Moreover, $  h^*(g) = \max_{y \in \R^n} \la g,y\ra - h(y)= \max_{y \in\R^n} \la g,y\ra - \psi_m^*(-y)=  \psi_m(-g)$, 
thus $y^t=-\nabla \psi_m(-g^t)$ is the negative gradient of $\psi_m$ at  point $-g^t$. Therefore, Algorithm \ref{calg:rand} does not explicitly depend on the function $h$, even though the
above analysis does.

Each iteration of Algorithm~\ref{calg:rand} updates only a randomly selected block $i_t$ in \eqref{c:algo3}, making it especially favorable 
when the number of blocks $m$ is large. However, similar difficulty as mentioned in  Section \ref{sec_SGD_alg} also appears 
when we integrate this algorithm with proximal-point type method to yield the final RapDual method in Algorithm \ref{calg_non}. Firstly, 
Algorithm~\ref{calg:rand} also keeps a few intertwined primal and dual sequences, thus we need to carefully decide the input and output of 
Algorithm~\ref{calg:rand} so that information from previous iterations of RapDual is fully used. Secondly,
the number of iterations performed by Algorithm~\ref{calg:rand} to solve each subproblem plays a vital role in the convergence rate of RapDual, which should be carefully predetermined.
 
Algorithm \ref{calg_non} describes the basic scheme of RapDual. At the beginning, all the blocks are initialized using the output from solving 
the previous subproblem. Note that $x_m^0$ is used to initialize $g$, which further helps compute the dual variable $y$ without using the conjugate function $h$ of $\psi_m$. 
We will derive 
the convergence result for Algorithm~\ref{calg:rand} in terms of primal variables and  construct  relations between successive search points $(\Bx^{\l},x_m^l)$,
which will be used to prove the final convergence of RapDual. 
\begin{algorithm}[H]  
\caption{RapDual for nonconvex multi-block  optimization}\label{calg_non}
\begin{algorithmic}
\STATE Compute $A_m^{-1}$ and reformulate problem~\eqref{c:problem_non} as \eqref{c:problem_convex}. 
\STATE{Let $\bBx^0\in X$, $\bx_m^0 \in\R^n$, such that $\bfA\bBx^0+\bx_m^0=\bfb $, and $\bar y^0=-\n f_m(\bar x^0_m)$.}
\FOR{$\ell = 1, \ldots, k$}
\STATE{Set $\Bx^{-1}=\Bx^0 = \bBx^{\l-1}$, $x_m^0=\bx_m^{\ell-1}$.}
\STATE{Run Algorithm \ref{calg:rand} with input $\Bx^{-1}$, $\Bx^{0}$, $x_m^0$ and $s$ to solve the following subproblem
}
\begin{align}\label{convex}
\min_{\Bx\in X,x_m\in \R^{n}} &~ \psi(\Bx)+\psi_m(x_m)\nn\\
\text{s.t. } &~\bfA\Bx+x_m =\bfb ,
\end{align}
to compute output $(\Bx^s, x_m^s)$,
where $\psi(\Bx)\equiv\psi^{\l}(\Bx):= f(\Bx)+ \mu\|\Bx-\bBx^{\l-1}\|^2$ and $\psi_m(x)\equiv\psi_m^{\l}(x):= f_m(x_m)+ \mu\|x_m-\bx_m^{\l-1}\|^2$.
\STATE{Set $\bBx^{\l}= \Bx^s$, $\bx_m^{\l}=x_m^{s}$.}
\ENDFOR
\RETURN $(\bBx^{\hat\l}, \bx_m^{\hat\l})$ for some random $\hat\l\in [k]$.
\end{algorithmic}
\end{algorithm}

\begin{algorithm}[H]
\caption{RaDual for solving subproblem \eqref{csubprob}}
\label{calg:rand}
\begin{algorithmic}
\STATE{Let $\Bx^{-1}=\Bx^0\in X$,  $x_m\in\R^{n}$, number of iterations $s$
and nonnegative parameters $\{\a_t\}$, $\{\tau_t\}$, $\{\eta_t\}$ be given. Set $g^0 = -x_m^0$.}
\FOR{$t = 1, \ldots, s$}
\STATE{1. Generate a random variable $i_t$ uniformly distributed over $[m-1]$.}
\STATE{2. Update $x^t$ and $y^t$ according to}
\begin{align}
\tilde{\Bx}^t =&~ \a_t (\Bx^{t-1}-\Bx^{t-2}) + \Bx^{t-1},\label{c:algo1}\\
g^t=&~(\tau_t g^{t-1}+\bfA\tilde{\Bx}^t-\bfb )/(1+\tau_t), \label{c:algo2-0}\\
y^t =&~\textstyle\argmin\limits_{y \in \mathbb{R}^n}~ h(y)+\la -\bfA\tilde{\Bx}^t+\bfb ,y\ra + \tau_t V_{h}(y,y^{t-1})=-\nabla \psi_m(-g^t),\label{c:algo2}\\
x_i^t = &~
\left\{
\begin{array}{ll}
\argmin_{x_i\in X_i}~  \psi_i(x_i) + \la \bfA_{i}^{\top} y^t,x_i\ra + \tfrac{\eta_t}{2}\|x_i-x_i^{t-1}\|^2, & i = i_t,\\\label{c:algo3}
x_i^{t-1}, & i \neq i_t.
\end{array}
\right.
\end{align}
\ENDFOR
\STATE{Compute $x_m^s=\argmin_{x_m\in{\R^{n}}} \{\psi_m(x_m)+\la x_m,y^s \ra\}$.}
\RETURN $(\Bx^s, x_m^s)$.
\end{algorithmic}
\end{algorithm}

We first define an approximate stationary point for problem \eqref{c:problem_non} before establishing the convergence of RapDual.
\begin{definition}\label{def 2}
A point $(\eBx,x_m) \in X\times\R^{n}$ is called an $(\e,\d,\s)$-solution of \eqref{c:problem_non} if there exists some $\hat \eBx\in X$, and $\lbd\in \R^n$ such that
\begin{eqnarray*}
\left[d(\nabla f(\hat \eBx)+ \bfA^{\top}\lbd, -N_{X} (\hat \eBx))\right]^2\leq \e, &\|\nabla f_m(x_m)+\lbd\|^2\leq \e , \\
\|\eBx-\hat \eBx\|^2\leq \d, &\|\bfA\eBx+x_m-\bfb \|^2\leq \s.
\end{eqnarray*}
A stochastic counterpart is one that satisfies 
\begin{eqnarray*}
\Eb\left[d(\nabla f(\hat \eBx)+ \bfA^{\top}\lbd, -N_{X} (\hat \eBx))\right]^2\leq \e, &\Eb\|\nabla f_m(x_m)+ \lbd\|^2\leq \e , \\
\Eb\|\eBx-\hat \eBx\|^2\leq \d, &\Eb \|\bfA\eBx+x_m-\bfb \|^2\leq \s.
\end{eqnarray*}
\end{definition}
Consider the unconstrained problem with $X=\R^{\sum_{i=1}^{m-1}d_i}$. If $(\Bx,x_m) \in X\times\R^{n}$ is an $(\e,\d,\s)$-solution with $\delta = \e/L^2$, then  exists some $\hat\Bx\in X$ such that $\|\nabla f(\hat \Bx)\|^2\leq \e$ and $\|\Bx-\hat \Bx\|^2\leq \d$. By similar argument in \eqref{quality}, we obtain $\|\nabla f(\Bx)\|^2\leq 4\e$. Besides, the definition of a $(\e,\d,\s)$-solution guarantees  $\|\nabla f_m(x_m)+\lbd\|^2\leq \e$ and $\|\bfA\Bx+x_m-\bfb \|^2\leq \s$, which altogether justify that $(\Bx,x_m)$ is a reasonably good solution.
\begin{theorem} \label{Thm_cmain}
Let the iterates $(\eBx^{\l},x_m^{\l})$ for $\l = 1, \ldots, k$ be generated by Algorithm \ref{calg_non} and $\hat \l$ be randomly selected from $[k]$. Suppose in Algorithm \ref{calg:rand}, number of iterations $s=\left\lceil-\log \wh\M/\log \a\right\rceil$ with
\beq\label{cparameters1}
\wh \M=(2+\tfrac{L}{\mu})\cdot\max\left\{2,\tfrac{L^2}{\mu^2}\right\}, \quad \a=1-\tfrac{2}{(m-1)(\sqrt{1+8c}+1)},\quad  c = \tfrac{\bar A^2}{\mu\bar{\mu}}=\tfrac{(2\mu+L)\bar A^2}{\mu}, \quad \bar{A} = \max_{i\in[m-1]} \|\bfA_{i}\|, 
\eeq
and other parameters are set to   
\beq\label{cparameters2}
\tsy{\a_t = (m-1)\a,\quad \g_t = \a^{-t},  
\quad \tau_t =  \tfrac{\a}{1-\a},\quad\eta_t = \tfrac{\left(\a-\tfrac{m-2}{m-1}\right)\mu}{1-\a},\quad \forall t = 1, \ldots, s .}
\eeq
Then there exists some $\lbd^*\in \R^n$ such that 
\begin{align*}
\Eb\left[d(\nabla f(\eBx_*^{\hat\l})+\bfA^{\top}\lbd^*,-N_{X}(\eBx_*^{\hat\l}))\right]^2&~\leq \tfrac{8\mu}{k}\left\{f(\bar\eBx^{0})+f_m(\bx_{m}^{0})-[f(\eBx^{*})+f_m(x_m^{*})]\right\},\\
\Eb\|\nabla f_m(x_m^{\hat\l})+\lbd^*\|^2&~\leq \tfrac{34\mu}{k}\left\{f(\bar\eBx^{0})+f_m(\bx_{m}^{0})-[f(\eBx^{*})+f_m(x_m^{*})]\right\},\\
\Eb\|\eBx^{\hat\l}-\eBx_*^{\hat\l}\|^2	&~\leq \tfrac{2\mu}{kL^2}\left\{f(\bar\eBx^{0})+f_m(\bx_{m}^{0})-[f(\eBx^{*})+f_m(x_m^{*})]\right\},\\
\Eb\|\bfA\eBx^{\hat\l}+x_m^{\hat\l}-\bfb \|^2&~ \leq \tfrac{2(\|\bfA\|^2+1)\mu}{kL^2} \left\{f(\bar\eBx^{0})+f_m(\bx_{m}^{0})-[f(\eBx^{*})+f_m(x_m^{*})]\right\},\end{align*}
where $(\eBx^*,x_m^*)$ and  $(\eBx^{\l}_*,x_{m^*}^{\l})$ denote the  optimal solutions to \eqref{c:problem_non} and the $\l$-th subproblem \eqref{csubprob},respectively.
\end{theorem}

Theorem \ref{Thm_cmain} ensures that our output solution $(\Bx^{\hat \l},x_m^{\hat \l})$ is close enough to 
an unknown approximate stationary point $(\Bx_*^{\hat \l},x_{m^*}^{\hat \l})$. According to Theorem~\ref{Thm_cmain},
we can bound the complexity of RapDual to compute a 
stochastic $(\e,\d,\s)$-solution of \eqref{c:problem_non}  in terms of block updates in \eqref{c:algo3}. 
Note that for  each subproblem \eqref{csubprob}, we only need to update $s$ primal blocks with
\[\tsy{s= \left\lceil-\tfrac{\log \wh\M}{\log \a}\right\rceil \sim\O\left( m\bar A\sqrt{\tfrac{L}{\mu}}\log\left(\tfrac{L}{\mu}\right)\right)}.\]
Let $\mathcal{D}^0:=f(\bBx^{0})+f_m(\bx_m^0)-[f(\Bx^{*})+f_m(x_m^{*})]$. It can be seen that 
the total number of primal block updates required to obtain a stochastic 
$(\e,\d,\s)$-solution can be bounded by 
\begin{equation} \label{comp_RapDual}
N(\e,\d,\s):=\tsy{\O\left( m\bar A\sqrt{L\mu}\log\left(\tfrac{L}{\mu}\right)\cdot \max\left\{\tfrac{1}{\e}, ~\tfrac{1}{\d L^2},~\tfrac{\|\bfA\|^2}{\s L^2}\right\}\mathcal{D}^0\right)}.
\end{equation}
As a comparison, the batch version of this algorithm would update all the $x_i^t$ for $i=1,\ldots,m$, in \eqref{c:algo3}, 
and thus would require
\[\hat N(\e,\d,\s):=\tsy{\O\left( m\|\bfA\|\sqrt{L\mu}\log\left(\tfrac{L}{\mu}\right)\cdot \max\left\{\tfrac{1}{\e}, ~\tfrac{1}{\d L^2},~\tfrac{\|\bfA\|^2}{\s L^2}\right\}\mathcal{D}^0\right)}.\]
primal block updates to obtain an $(\e,\d,\s)$-solution of \eqref{c:problem_non}.
Therefore, the benefit of  randomization comes from the difference between  $\|\bfA\|$ and $\bar A$. 
Obviously we always have $\|\bfA\|> \bar A$,  and the relative gap between $\|\bfA\|$ and $\bar A$ can be large when all the matrix blocks have close norms. 
 In the case when all the blocks are identical, i.e. $\bfA_1 = \bfA_2=\ldots=\bfA_{m-1}$, we immediately have $\|\bfA\| = \sqrt{m-1}\bar{A}$, which means 
 that RapDual can 
 potentially save the number of primal block updates by a factor of ${\cal O}(\sqrt{m})$ than its batch counterpart.
 
 It is also interesting to compare RapDual with the nonconvex randomized block coordinate descent method in \cite{DangLan2015}.
 To compare these methods, let us assume that $f_i$ is smooth with $\bar L$-Lipschitz continuous gradient 
 for some $\bar L \ge \mu$ for any $i = 1, \ldots, m$.
 Also let us assume that $\s > \|\bfA\|^2 \e / L^2$, $\d = \e / L^2$, and $X=\R^{\sum_{i=1}^{m-1}d_i}$.  Then,
  after disregarding some constant factors, the bound in \eqref{comp_RapDual} reduces
 to ${\cal O}(m \bar A \sqrt{L \mu} \mathcal{D}^0 / \epsilon)$, which is always smaller than
 the ${\cal O}(m (\bar L + L \bar A^2) \mathcal{D}^0 / \epsilon)$ complexity bound 
 implied by Corollary 4.4 of \cite{DangLan2015}.


\begin{remark} \label{remarkinverse}
In this paper, we assume that $A_m^{-1}$ is easily computable. One natural question is whether we can avoid the computation of $A_m^{-1}$ 
by directly solving \eqref{c:problem_non} instead of its reformulation \eqref{c:problem_convex}.
To do so, we can iteratively solve the following saddle-point subproblems in place of the ones in \eqref{eqn:saddle}:
\begin{align}\label{eqn:saddle_original}
&\min_{\Bx\in X,x_m\in \R^{n}}   \{ \psi(\Bx)+\psi_m(x_{m}) + \max_{y\in \mathbb{R}^{n}} \lla \tsum_{i=1}^{m} A_ix_i - b , y\rra \}
\nn \\
=&\min_{\Bx\in X}  \{\psi(\Bx) + \max_{y\in \mathbb{R}^{n}} [\lla A\Bx- b , y\rra+ \min_{x_m\in\R^{n}}\left\{\psi_m(x_m)+\la A_m x_m,y\ra\right\}]\}\nn\\
=&\min_{\Bx\in X}  \{ \psi(\Bx) + \max_{y\in \mathbb{R}^{n}} [\lla A\Bx-b , y\rra-\tilde{h}(y) ]\},
\end{align}
where $A:=[A_1, \ldots, A_{m-1}]$ and $\tilde{h}(y) := \max_{x_m\in\R^{n}}\left\{-\psi_m(x_m)-\la A_m x_m,y\ra\right\}$.
Instead of keeping $\tilde{h}(y)$ in the projection subproblem~\eqref{c:algo2}
as we did for $h(y)$, we need to linearize it at each iteration by computing its gradients
$\nabla \tilde h(y^{t-1}) = - A_m^T \bar x_m(y^{t-1})$, where $\bar x_m(y^{t-1}) = \argmax_{x_m\in\R^{n}}\left\{-\psi_m(x_m)-\la A_m x_m,y^{t-1}\ra\right\}$.
Note that the latter optimization problem can be solved by using an efficient first-order method due to the smoothness and strong concavity of its objective function. 
As a result, we will be able to obtain a similar rate of convergence to RapDual without computing $A_m^{-1}$.
However, the statement and analysis
of the algorithm will be much more complicated than RapDual in its current form.
\end{remark}

\subsection{Convergence analysis for RapDual}\label{sec_cSGD_analysis}
In this section, we first show the convergence of Algorithm~\ref{calg:rand} for solving the convex multi-block subproblem \eqref{convex} with
\begin{enumerate}[label=(\Alph*)]
\item $\psi_i(x)-\psi_i(y)-\la\nabla \psi_i(y),x-y\ra \geq \tfrac{\mu}{2}\|x-y\|^2, \ \ \forall x, y\in X_i, \quad i=1,\ldots,m-1,$
\item $\tfrac{\mu}{2}\|x-y\|^2\leq \psi_m(x)-\psi_m(y)-\la\nabla \psi_m(y),x-y\ra \leq \tfrac{\hat L}{2}\|x-y\|^2, \ \ \forall x, y\in \R^n.$
\end{enumerate}

Some simple relations about the iterations generated by the Algorithm \ref{calg:rand} are characterized in the following lemma, and the proof follows directly
from the definition of $\hat \Bx$ in \eqref{c:xfull}, thus has been omitted.  
\begin{lemma}\label{clemma:full}
Let $\hat{\eBx}^0 = \eBx^0$ and $\hat{\eBx}^t$ for $t = 1,\ldots, s$ be defined as follows:
\begin{align}\label{c:xfull}
\hat{\eBx}^t = \argmin_{\eBx\in X}~  \psi(\eBx) + \la \bfA^{\top} y^t,\eBx\ra + \tfrac{\eta_t}{2}\|\eBx-\eBx^{t-1}\|^2,
\end{align}
where $\eBx^t$  and $y^t$ are obtained from \eqref{c:algo2}-\eqref{c:algo3}, then
we have
\begin{align}
\Eb_{i_t}\left\{\|\eBx-\hat{\eBx}^t\|^2\right\} &= \Eb_{i_t}\left\{(m-1)\|\eBx-\eBx^t\|^2-(m-2)\|\eBx-\eBx^{t-1}\|^2\right\},\label{c:full2}\\
\Eb_{i_t}\left\{\|\hat{\eBx}^t-\eBx^{t-1}\|^2\right\} &= \Eb_{i_t}\left\{(m-1)\|x_{i_t}^t-x_{i_t}^{t-1}\|^2\right\}\label{c:full3}.
\end{align}
\end{lemma}

The following lemma \ref{lemma 5}  builds some  connections between the input and output of  Algorithm \ref{calg:rand} in terms of both primal and dual variables, 
and the proof can be found in Appendix \ref{pf_lemma5}.
\begin{lemma}\label{lemma 5}
Let the iterates $\eBx^t$ and $y^t$ for $t = 1, \ldots, s$ be generated by Algorithm \ref{calg:rand}
and $(\eBx^*, y^*)$ be a saddle point of \eqref{eqn:saddle}.
Assume that the parameters in Algorithm \ref{calg:rand} satisfy for all $t = 1, \ldots, s-1$
\begin{align}
\a_{t+1} = &~ (m-1) \tilde{\a}_{t+1},\label{c:ss1}\\
\g_t = &~\g_{t+1}\tilde{\a}_{t+1},\label{c:ss2}\\
\g_{t+1}\left((m-1)\eta_{t+1} + (m-2)\mu\right) \le &~(m-1)\g_t(\eta_t + \mu),\label{c:ss3}\\
\g_{t+1}\tau_{t+1} \le &~ \g_t(\tau_t+1),\label{c:ss4}\\
2(m-1)\tilde{\a}_{t+1} \bar{A}^2 \le &~ \bar\mu\eta_t\tau_{t+1},\label{c:ss5}
\end{align}
where $\bar{A} = \max_{i\in [m-1]} \|\bfA_{i}\|$.
Then we have
\begin{align}\label{eq: gap-det}
\Eb_s&\tsy{\left\{\tfrac{\g_1((m-1)\eta_1+(m-2)\mu)}{2}\|\eBx^*-\eBx^0\|^2 - \tfrac{(m-1)\g_s(\eta_s+\mu)}{2}\|\eBx^*-\eBx^s\|^2\right\}}\nn\\
&~+ \Eb_s\tsy{\left\{\g_1\tau_1V_{  h}(y^*,y^0) - \tfrac{\g_s(\tau_s+1)\bar\mu}{2}V_{  h}(y^*,y^s)\right\}\geq 0}.
\end{align}
\end{lemma}

Now we present the main convergence result of Algorithm \ref{calg:rand} in Theorem \ref{theo3}, which eliminates the dependence on dual variables 
and relates directly the successive searching points of RapDual.
\begin{theorem}\label{theo3}
Let $(x^*, y^*)$ be a saddle point of \eqref{eqn:saddle}, and suppose that the parameters $\{\a_t\}$, $\{\tau_t\}$, $\{\eta_t\}$ and $\{\g_t\}$ are set as in \eqref{cparameters1} and \eqref{cparameters2}, and $\tilde\a_t = \a$.
Then, for any $s \ge 1$, we have
\begin{align*}
  \Eb_s\left\{\|\eBx^s-\eBx^*\|^2+\|x_m^s-x_m^*\|^2\right\}\leq 
\a^s\M(\|\eBx^0-\eBx^*\|^2 +\|x_m^0-x_m^*\|^2),
\end{align*}
where $x_m^*=\argmin_{x_m\in{\R^{n}}} \{\psi_m(x_m)+\la x_m,y^* \ra\}$ and $\M=2\hat L/\mu$.
\end{theorem}
\begin{proof}
It is easy to check that \eqref{cparameters1} and \eqref{cparameters2} satisfy conditions \eqref{c:ss1}, \eqref{c:ss2}, \eqref{c:ss3} \eqref{c:ss4}, and \eqref{c:ss5} when $\mu,\bar{\mu} > 0$.
Then we have 
\[
\tsy{\Eb_s\left\{\tfrac{(m-1)\g_s(\eta_s+\mu)}{2}\|\Bx^s-\Bx^*\|^2+\tfrac{\g_s(\tau_s+1)\bar\mu}{2}V_{  h}(y^s,y^*)\right\} 
  \leq 
\tfrac{\g_1((m-1)\eta_1+(m-2)\mu)}{2}\|\Bx^0-\Bx^*\|^2 +\g_1\tau_1V_{  h}(y^0,y^*)}.
\]
Therefore, by plugging in those values in  \eqref{cparameters1} and \eqref{cparameters2}, we have
\begin{equation}\label{c:p1}
\Eb_s\left[\mu\|\Bx^s-\Bx^*\|^2 +V_{  h}(y^s,y^*)\right]
\le  \a^s
\left[\mu\|\Bx^0-\Bx^*\|^2 +2V_{  h}(y^{0},y^*)
\right],
\end{equation}
Since $h(y)$ has $1/\mu$-Lipschitz continuous gradients and is  $1/L$-strongly convex, we obtain
\begin{align}
&V_{  h}(y^s,y^*)\geq \tfrac{\mu}{2}\|\nabla  h(y^s)-\nabla  h(y^*)\|^2=\tfrac{\mu}{2}\|-x_m^s+x_m^*\|^2\label{d:dist-u},\\
&V_{  h}(y^0,y^*)\leq \tfrac{\hat L}{2}\|\nabla  h(y^0)-\nabla  h(y^*)\|^2=\tfrac{\hat L}{2}\|-x_m^0+x_m^*\|^2\label{d:dist-uy}.
\end{align}
Combining \eqref{c:p1}, \eqref{d:dist-u} and \eqref{d:dist-uy}, we have
\[\Eb_s\left\{\|\Bx^s-\Bx^*\|^2+\|x_m^s-x_m^*\|^2\right\}\leq 
\a^s\M
(\|\Bx^0-\Bx^*\|^2 +\|x_m^0-x_m^*\|^2).\]
\end{proof}   

The above theorem shows that subproblem \eqref{convex} can be solved efficiently by Algorithm \ref{calg:rand} with a linear rate of convergence.
In fact, we need not solve it too accurately. With a fixed and relatively small number of iterations $s$ Algorithm \ref{calg:rand} can still converge, as shown by the following lemma.

\begin{lemma}\label{clemma_final}
Let the inner iteration number $s\geq \left\lceil-\log \M/\log \a\right\rceil$ with $\M=4+2L/\mu$ be given.
Also the iterates $(\eBx^{\l},x_m^{\l})$ for $\l = 1, \ldots, k$ be generated by Algorithm \ref{calg_non} and $\hat \l$ be randomly selected from $[k]$. Then 
\begin{align*}
\Eb\left(\|\Bx^{\l}_*-\bBx^{\l-1}\|^2+\|x_{m^*}^{\l}-\bx_m^{\l-1}\|^2\right)
& \le
\tfrac{1}{k\mu(1-M\a^s)}\left\{f(\bBx^{0})+f_m(\bx_{m}^{0})-[f(\Bx^{*})+f_m(x_m^{*})]\right\},\\
\Eb\left(\|\Bx^{\l}_*-\bBx^{\l}\|^2+\|x_{m^*}^{\l}-\bx_m^{\l}\|^2\right)
& \le
\tfrac{\M\a^s}{k\mu(1-M\a^s)}\left\{f(\bBx^{0})+f_m(\bx_{m}^{0})-[f(\Bx^{*})+f_m(x_m^{*})]\right\},
\end{align*}
where  $(\eBx^*,x_m^*)$ and $(\eBx^{\l}_*, x_{m^*}^{\l})$ are the optimal solutions to \eqref{c:problem_non} and the $\l$-th subproblem \eqref{csubprob}, respectively.
\end{lemma}	
\begin{proof}
According to Theorem \ref{theo3}, we have
\beq\label{c_non1}
\Eb\left(\|\bBx^{\l}-\Bx_*^{\l}\|^2+\|\bx_m^{\l}-x_{m^*}^{\l}\|^2\right)\leq 
  \a^s\M(\|\bBx^{\l-1}-\Bx^{\l}_*\|^2+\|\bx_m^{\l-1}-x_{m^*}^{\l}\|^2).
\eeq
Let us denote $(\Bx^{0}_*,x_{m^*}^{0})$ = $(\bBx^0, \bx_{m}^{0})$ and by selection, it is feasible to subproblem \eqref{csubprob} when $\l=1$.
Since $(\Bx^{\l}_*,x_{m^*}^{\l})$ is optimal and $(\Bx_*^{\l-1}, x_{m^*}^{\l-1})$ is feasible to the $\l$-th subproblem, we have
\[\psi^{\l}(\Bx_*^{\l})+\psi_m^{\l}(x_{m^*}^{\l})\leq \psi^{\l}(\Bx_*^{\l-1})+\psi_m^{\l}(x_{m^*}^{\l-1}).\]
Plugging in the definition of $\psi^{\l}$ and $\psi_m^{\l}$ in the above inequality, and summing up from $\l=1$  to $k$, we have 
\begin{align}\label{c_non2}
\sl [f(\Bx^{\l}_*)+f_m(x_{m^*}^{\l})+&\mu(\|\Bx^{\l}_*-\bBx^{\l-1}\|^2+\|x_{m^*}^{\l}-\bx_m^{\l-1}\|^2)]
 \le \nn\\
 &\sl [f(\Bx^{\l-1}_*)+f_m(x_{m^*}^{\l-1})+\mu (\|\Bx_{*}^{\l-1}-\bBx^{\l-1}\|^2+\|x_{m^*}^{\l-1}-\bx_m^{\l-1}\|^2)].
\end{align}
Combining \eqref{c_non1} and \eqref{c_non2} and noticing that $(\Bx^{0}_*,x_{m^*}^{0})$ = $(\bBx^0, \bx_{m}^{0})$, we have
\begin{align}\label{c_non3}
\mu\sl\Eb(\|\Bx^{\l}_*-\bBx^{\l-1}\|^2+\|x_{m^*}^{\l}-\bx_m^{\l-1}\|^2)&\leq \sl\{f(\Bx^{\l-1}_*)+f_m(x_{m^*}^{\l-1})-[f(\Bx^{\l}_*)+f_m(x_{m^*}^{\l})]\}\nn\\
&~~~+\mu\sl\Eb(\|\Bx^{\l}_*-\bBx^{\l}\|^2+\|x_{m^*}^{\l}-\bx_m^{\l}\|^2) \nn\\
&\leq f(\bBx^{0})+f_m(\bx_{m}^{0})-[f(\Bx^{*})+f_m(x_m^{*})]\nn \\
&~~~+\mu\a^s \M\sl\Eb(\|\Bx^{\l}_*-\bBx^{\l-1}\|^2+\|x_{m^*}^{\l}-\bx_m^{\l-1}\|^2).
\end{align}
In view of \eqref{c_non3} and \eqref{c_non1}, we have
\begin{align*}
\sl\Eb\left(\|\Bx^{\l}_*-\bBx^{\l-1}\|^2+\|x_{m^*}^{\l}-\bx_m^{\l-1}\|^2\right)
& \le
\tfrac{1}{\mu(1-M\a^s)}\left\{f(\bBx^{0})+f_m(\bx_{m}^{0})-[f(\Bx^{*})+f_m(x_m^{*})]\right\}\\
\sl\Eb\left(\|\Bx^{\l}_*-\bBx^{\l}\|^2+\|x_{m^*}^{\l}-\bx_m^{\l}\|^2\right)
& \le
\tfrac{\M\a^s}{\mu(1-M\a^s)}\left\{f(\bBx^{0})+f_m(\bx_{m}^{0})-[f(\Bx^{*})+f_m(x_m^{*})]\right\},
\end{align*}
which, in view of the fact that $\hat\l$ is chosen randomly in $[k]$, implies our results.
\end{proof}

\vgap

Now we are ready to prove the results in Theorem \ref{Thm_cmain} with all the results proved above.

\noindent {\bf Proof of Theorem  \ref{Thm_cmain}.}
By the optimality condition of the $\hat\l$-th subproblem \eqref{csubprob}, there exists some $\lbd^*$ such that 
\begin{align}             
&\nabla \psi^{\hat\l}(\Bx_*^{\hat\l})+\bfA^{\top}\lbd^*\in -N_{X}(\Bx_*^{\hat\l}),\nn \\
& \nabla \psi_m^{\hat\l}(x_{m^*}^{\hat\l})+\lbd^*=0, \nn \\ 
&\bfA\Bx_*^{\hat\l}+x_{m^*}^{\hat\l}=\bfb .\label{feasibility} 
\end{align}
Plugging in the definition of $\psi^{\hat\l}$ and $\psi_m^{\hat\l}$, we have 
\begin{align}
&\nabla f^{\hat\l}(\Bx_*^{\hat\l})+2\mu(\Bx_*^{\hat\l}-\bBx^{\hat{\l}-1})+\bfA^{\top}\lbd^*\in -N_{X}(\Bx_*^{\hat\l}),\label{copt_x}\\
&\nabla f_m^{\hat\l}(x_{m^*}^{\hat\l})+2\mu(x_{m^*}^{\hat\l}-\bx_m^{\hat{\l}-1})+\lbd^*=0.\label{copt_xm}
\end{align}
Now we are ready to evaluate the quality of the solution $(\bBx^{\hat\l},\bx_m^{\hat\l})$. In view of \eqref{copt_x} and Lemma 
\ref{clemma_final}, we have
\begin{align*}
\Eb\left[d(\nabla f^{\hat\l}(\Bx_*^{\hat\l})+\bfA^{\top}\lbd^*, -N_{X}(\Bx_*^{\hat\l}))\right]^2&\leq \Eb\|2\mu(\Bx_*^{\hat\l}-\bBx^{\hat{\l}-1})\|^2\\
& \leq \tfrac{4\mu}{k(1-\M\a^s)}\left\{f(\bBx^{0})+f_m(\bx_m^{0})-[f(\Bx^{*})+f_m(x_m^{*})]\right\}\\
& \leq \tfrac{8\mu}{k}\left\{f(\bBx^{0})+f_m(\bx_m^{0})-[f(\Bx^{*})+f_m(x_m^{*})]\right\}.
\end{align*}
Similarly, due to  \eqref{copt_xm} and Lemma 
\ref{clemma_final}, we have 
\begin{align*}
\Eb\|\nabla f_m^{\hat\l}(x_m^{\hat\l})+\lbd^*\|^2&~=\Eb\|\nabla f_m^{\hat\l}(x_m^{\hat\l})-\nabla f_m^{\hat\l}(x_{m^*}^{\hat\l})-2\mu(x_{m^*}^{\hat\l}-x_m^{\hat{\l}-1})\|^2\\
&~\leq 2\Eb\{\|\nabla f_m^{\hat\l}(x_m^{\hat\l})-\nabla f_m^{\hat\l}(x_{m^*}^{\hat\l})\|^2+4\mu^2\|x_{m^*}^{\hat\l}-x_m^{\hat{\l}-1}\|^2\}\\
&~\leq \Eb\left\{18\mu^2\|x_{m^*}^{\hat\l}-\bx_m^{\hat{\l}}\|^2+8\mu^2\|x_{m^*}^{\hat\l}-\bx_m^{\hat{\l}-1}\|^2\right\}\\
&~\leq \mu^2(8+18\M\a^s)\Eb\{\|\Bx^{\hat\l}-\bBx^{\hat{\l}-1}\|^2+\|x_{m^*}^{\hat\l}-\bx_m^{\hat{\l}-1}\|^2\}\\
&~\leq \tfrac{2\mu\left(4+9\M\a^s\right)}{k(1-\M\a^s)}   \Eb\left\{f(\bBx^{0})+f_m(\bx_{m}^{0})-[f(\Bx^{*})+f_m(x_m^{*})]\right\}\\
&~\leq \tfrac{34\mu}{k}\left\{f(\bBx^{0})+f_m(\bx_{m}^{0})-[f(\Bx^{*})+f_m(x_m^{*})]\right\}.
\end{align*}
By Lemma \ref{clemma_final} we have 
\begin{align*}
\Eb\|\Bx^{\hat\l}-\Bx_*^{\hat\l}\|^2	&\leq \tfrac{\M\a^s}{k\mu(1-\M\a^s)}\left\{f(\bBx^{0})+f_m(x_{m}^{0})-[f(\Bx^{*})+f_m(x_m^{*})]\right\}\\
&\leq \tfrac{2\M\a^s}{k\mu}\left\{f(\bBx^{0})+f_m(\bx_{m}^{0})-[f(\Bx^{*})+f_m(x_m^{*})]\right\}\\
&\leq \tfrac{2\mu}{kL^2}\left\{f(\bBx^{0})+f_m(\bx_{m}^{0})-[f(\Bx^{*})+f_m(x_m^{*})]\right\}.
\end{align*}
Combining \eqref{feasibility} and Lemma \ref{clemma_final}, we have 
\begin{align*}
\Eb\|\bfA\bBx^{\hat\l}+\bx_m^{\hat\l}-\bfb \|^2&=\Eb\|\bfA(\bBx^{\hat\l}-\Bx_*^{\hat\l})+\bx_m^{\hat\l}-x_{m^*}^{\hat\l}\|^2\\
&\leq 2\Eb\{\|\bfA\|^2\|\bBx^{\hat\l}-\Bx_*^{\hat\l}\|^2+\|\bx_m^{\hat\l}-x_{m^*}^{\hat\l}\|^2\}\\
&\leq 2 (\|\bfA\|^2+1) \Eb
\{\|\bBx^{\hat\l}-\Bx_*^{\hat\l}\|^2+\|\bx_m^{\hat\l}-x_{m^*}^{\hat\l}\|^2\}\\
&\leq \tfrac{2(\|\bfA\|^2+1)\M\a^s}{k\mu(1-\M\a^s)} \left\{f(\bBx^{0})+f_m(\bx_{m}^{0})-[f(\Bx^{*})+f_m(x_m^{*})]\right\}\\
& \leq\tfrac{2(\|\bfA\|^2+1)\mu}{kL^2}\left\{f(\bBx^{0})+f_m(\bx_{m}^{0})-[f(\Bx^{*})+f_m(x_m^{*})]\right\}.
\end{align*}
\endproof

\section{Numerical experiments} \label{sec-num}   
In this section, we report some preliminary numerical results for both RapGrad and RapDual and demonstrate their potential advantages 
in Subsection \ref{numerical 1} and \ref{numerical 2}, respectively.

\subsection{Nonconvex finite-sum optimization}\label{numerical 1}
We consider the least square problem with the smoothly clipped absolute deviation (SCAD) penalty as a testing problem. SCAD has been proved in \cite{fan2001variable} to be efficient in variable selection. While the original SCAD $p_{\lbd,\g}$ defined below does not have smooth gradient at $x=0$, we can bypass this potential problem by using a small positive 
number $\e$ to obtain a smooth approximation $p_{\lbd,\g,\e}$:
\[
p_{\lbd,\g}(x) = \left\{
\begin{aligned}
&\lbd|x|  && ~\text{if} ~|x|\leq \lbd,\\
&\tfrac{2\g\lbd|x|-x^2-\lbd^2}{2(\g-1)}  && ~\text{if}~ \lbd<|x|<\g \lbd,\\
&\tfrac{\lbd^2(\g+1)}{2}  && ~\text{if} ~|x|\geq\g\lbd,
\end{aligned}
\right.
\quad
p_{\lbd,\g,\e}(x) = \left\{
\begin{aligned}
&\lbd(x^2+\e)^{\frac12}  && ~\text{if} ~(x^2+\e)^{\frac12}\leq \lbd,\\
&\tfrac{2\g\lbd(x^2+\e)^{\frac12}-(x^2+\e)-\lbd^2}{2(\g-1)}  && ~\text{if}~ \lbd<(x^2+\e)^{\frac12}<\g \lbd,\\
&\tfrac{\lbd^2(\g+1)}{2}   && ~\text{if} ~(x^2+\e)^{\frac12}\geq\g\lbd,
\end{aligned}
\right.
\]
where $\g>2$, $\lbd >0$, and $\e>0$ are given. Using $p_{\lbd,\g,\e}$, our problem of interest is given by 
\[
\min_{x\in \R^n} \tfrac{1}{2m}\|Ax-b\|^2 + \tfrac{\rho}{2}\tsum_{i=1}^np_{\lbd,\g,\e}(x_i),\] 
which can be viewed as  a special case of problem \eqref{problem_non}  with $f_i(x) = \tfrac{1}{2}(a_i^{\top}x-b_i)^2 + \tfrac{\rho}{2}\sum_{i=1}^np_{\lbd,\g,\e}(x_i)$. 
Here $a_i$ denotes the $i$-th row of $A$.
It is easy to see that assumptions \eqref{assumptions_non1} and \eqref{assumptions_non2} are satisfied with $\mu=\rho/[2(\g-1)]$ 
and $L = \rho\lbd \e^{-1/2} / 2+\max_{1\leq i\leq m}\|a_i\|^2$.  Thus the condition number $L/\mu$ usually dominates $m$. 

We test Algorithm \ref{alg_non} (RapGrad), both randomized and deterministic batch versions, on some randomly generated data sets with dimension $m=1000$, $n=100$. 
Note all entries of matrix $A$ and $20$ uniformly chosen  components of $\hat x$  are i.i.d from N(0,1). The remaining variables of $\hat x$ are set to $0$ and $b$ is computed by $b=A \hat x$.
The parameters used in $p_{\lbd,\g,\e}$ are   $\e = \num{e-3}$, $\lbd = 2$, $\g = 4$ and the penalty $\rho$ is set to $0.01$. Notice that the $x$-axis represents the number of gradient evaluations divided by $m$, which is counted as number of passes to the dataset, i.e., each iteration of randomized  gradient computation is $1/m$ pass and the full-gradient computation of counts as $1$ pass.
As is shown by Figure \ref{fig 1}, randomized version can reduce both objective value and norm of gradient faster than its batch counterpart in terms of number of gradient evaluations.  
We also observe that there are some zig-zag pattern for the batch algorithm in Figure~\ref{fig 1}.b), which might have been related some numerical stability issue.
\begin{figure}[H].  
\centering
\begin{minipage}[t]{0.4\linewidth}
\centering
\includegraphics[width=\textwidth]{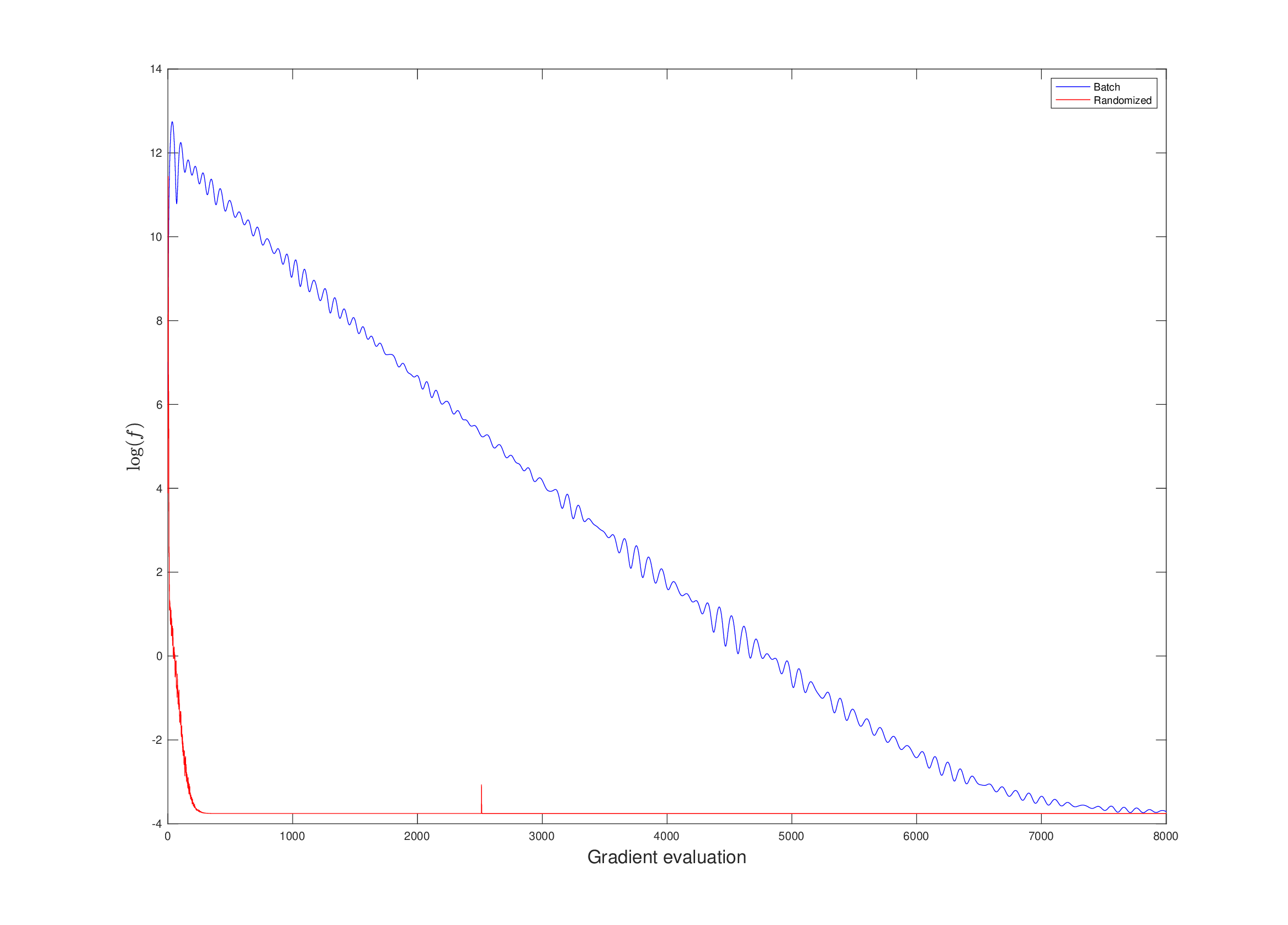}
\end{minipage}
\begin{minipage}[t]{0.4\linewidth}
\centering
\includegraphics[width=\textwidth]{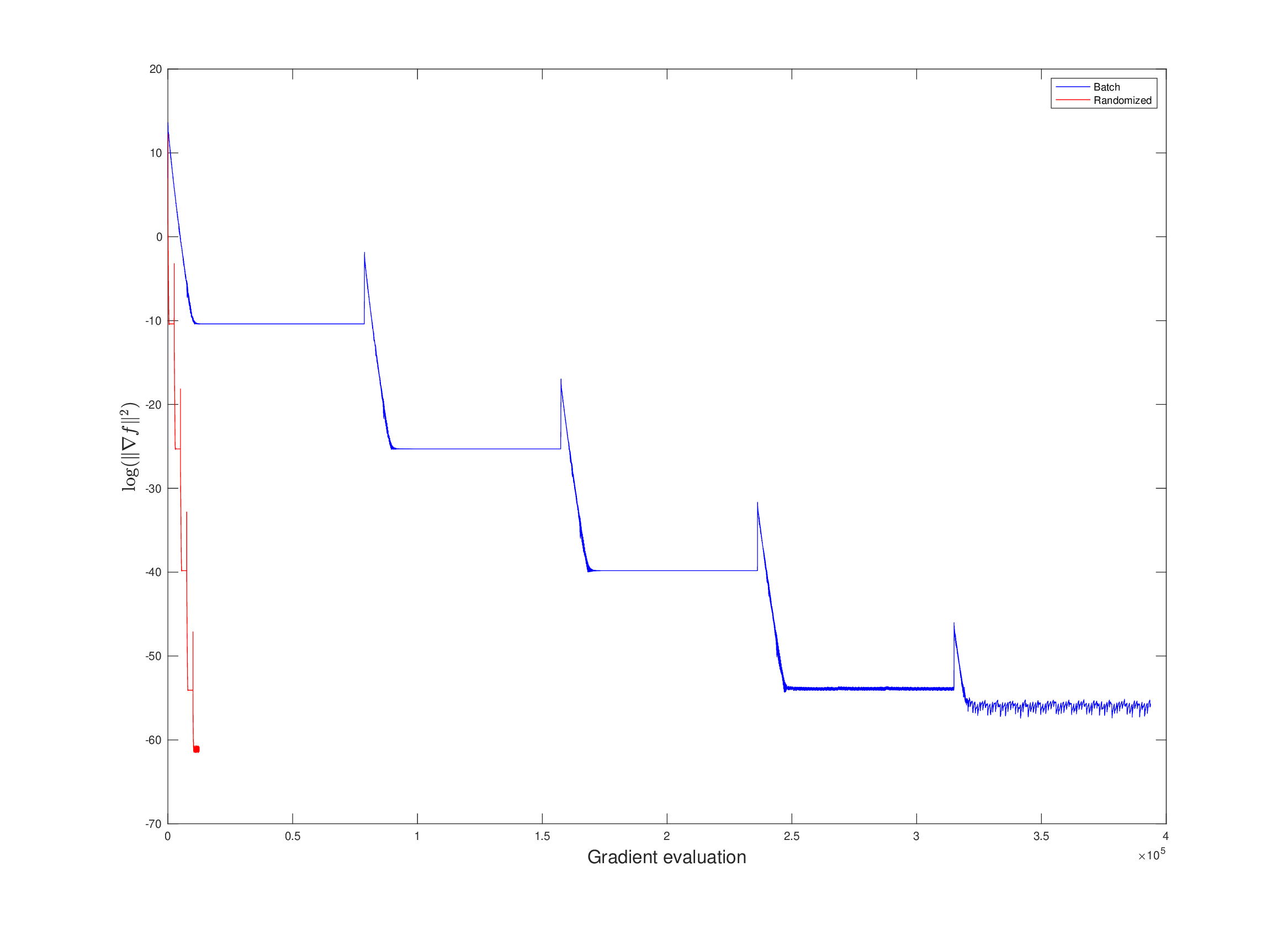}
\end{minipage}
\caption{Batch and randomized versions of  RapGrad on  SCAD-penalized least squares. Left Figure: Comparison on objective value $f$. Right Figure: Comparison on square of gradient norm $\|\n f\|^2$.}\label{fig 1}
\end{figure} 

We also compare our RapGrad with the full SVRG in non-convex setting (Algorithm 2 in \cite{allen2016variance}) and Accelerated Gradient method (AG) in \cite{ghadimi2016accelerated}.  Notice that, both RapGrad and SVRG are randomized algorithms, while AG is a deterministic batch method. For the sake of fairness in comparison, all the parameters in the three algorithms mentioned above are set to their theoretical values without any tuning in our experiments.
Figure \ref{fig_compare} shows our Algorithm \ref{alg_non} not only reduces the function value as well as gradient norm faster than both SVRG and AG.

\begin{figure}[H]
\centering
\begin{minipage}[t]{0.4\linewidth}
\centering
\includegraphics[width=\textwidth]{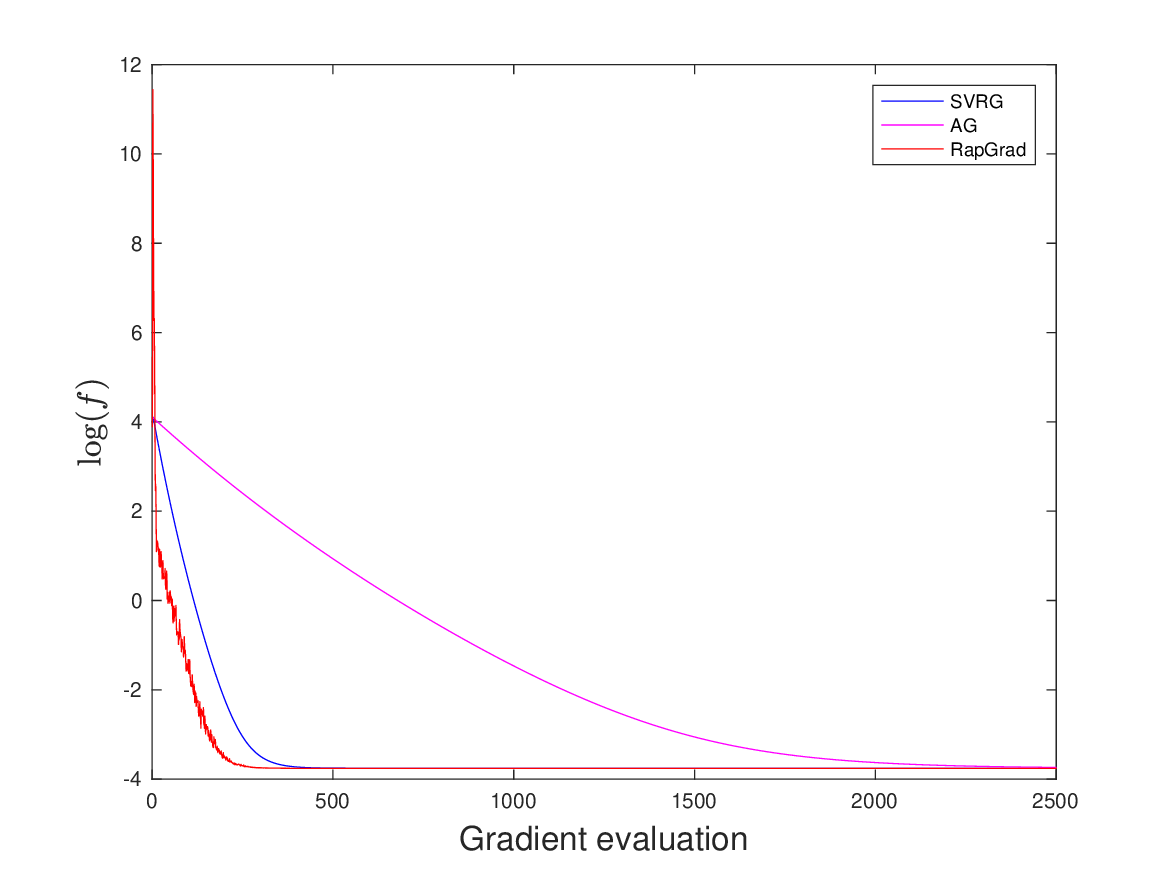}
\end{minipage}
\begin{minipage}[t]{0.4\linewidth}
\centering
\includegraphics[width=\textwidth]{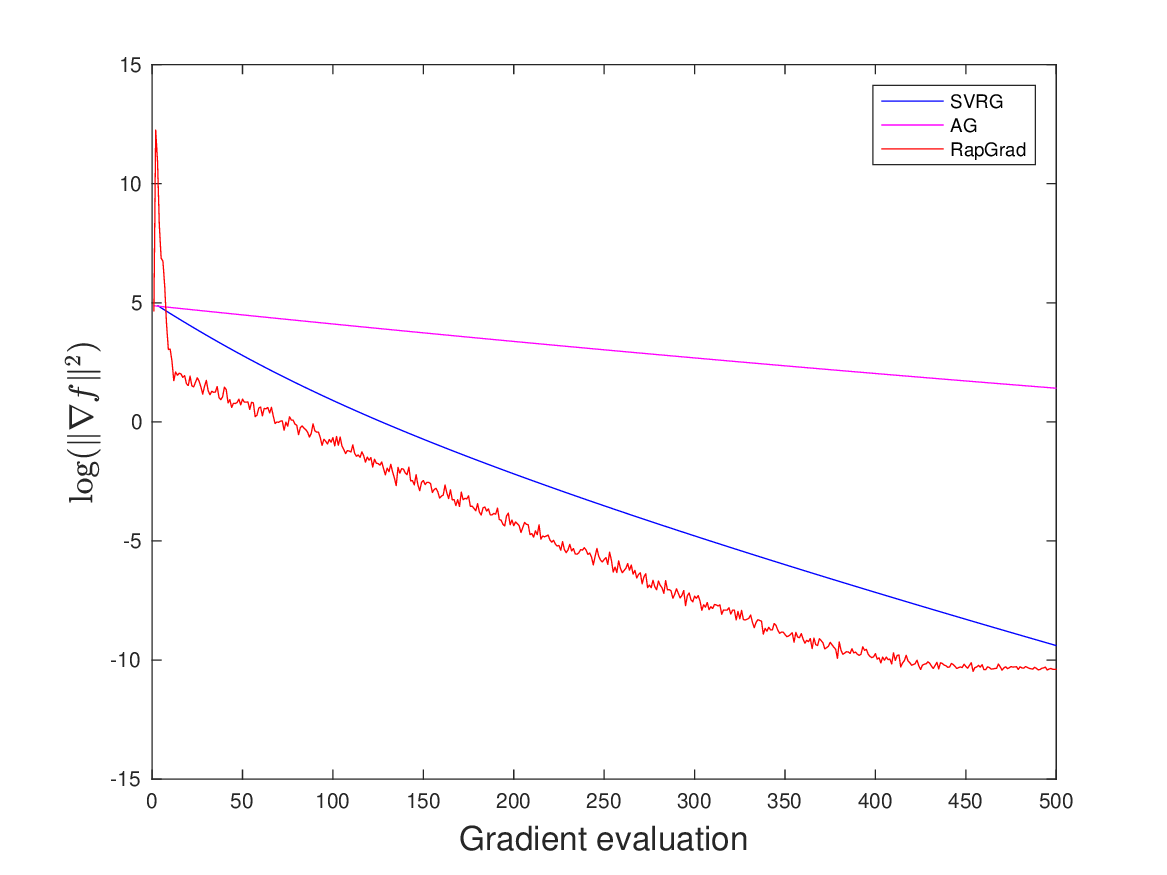}
\end{minipage}
\caption{Comparison on function value $f$ and square of gradient norm $\|\n f\|^2$ for Algorithm \ref{alg_non}, SVRG and AG. Left Figure: Comparison on function value $f$. Right Figure:  Comparison on square of gradient norm $\|\n f\|^2$.}
\label{fig_compare}
\end{figure}.

In fact, our estimate on $s=\left\lceil-\log (6M/5)/\log \a\right\rceil$ seems to be too pessimistic, 
and the subproblems are solved to unnecessarily high accuracy. As a result,  some spikes show in Figure \ref{fig 1}, 
which correspond to the occasions when an inner loop for solving the subproblem completes and a new search point is obtained to 
update the subproblem. Early termination of the inner loops may help to remove those spikes. Reducing the number of 
inner iterations may not guarantee above mentioned convergence rate theoretically, but may improve the practical performance
of RapGrad for this problems in our experiments. From Figure \ref{fig 2} and Figure \ref{fig 3}, we can conclude that, by using smaller $s$, 
our randomized algorithm is able to reduce $f$ and $\|\n f\|^2$ much faster, whereas the batch version converges faster 
in terms of the gradient norm.
  
\begin{figure}[H]
\centering
\begin{minipage}[t]{0.4\linewidth}
\centering
\includegraphics[width=\textwidth]{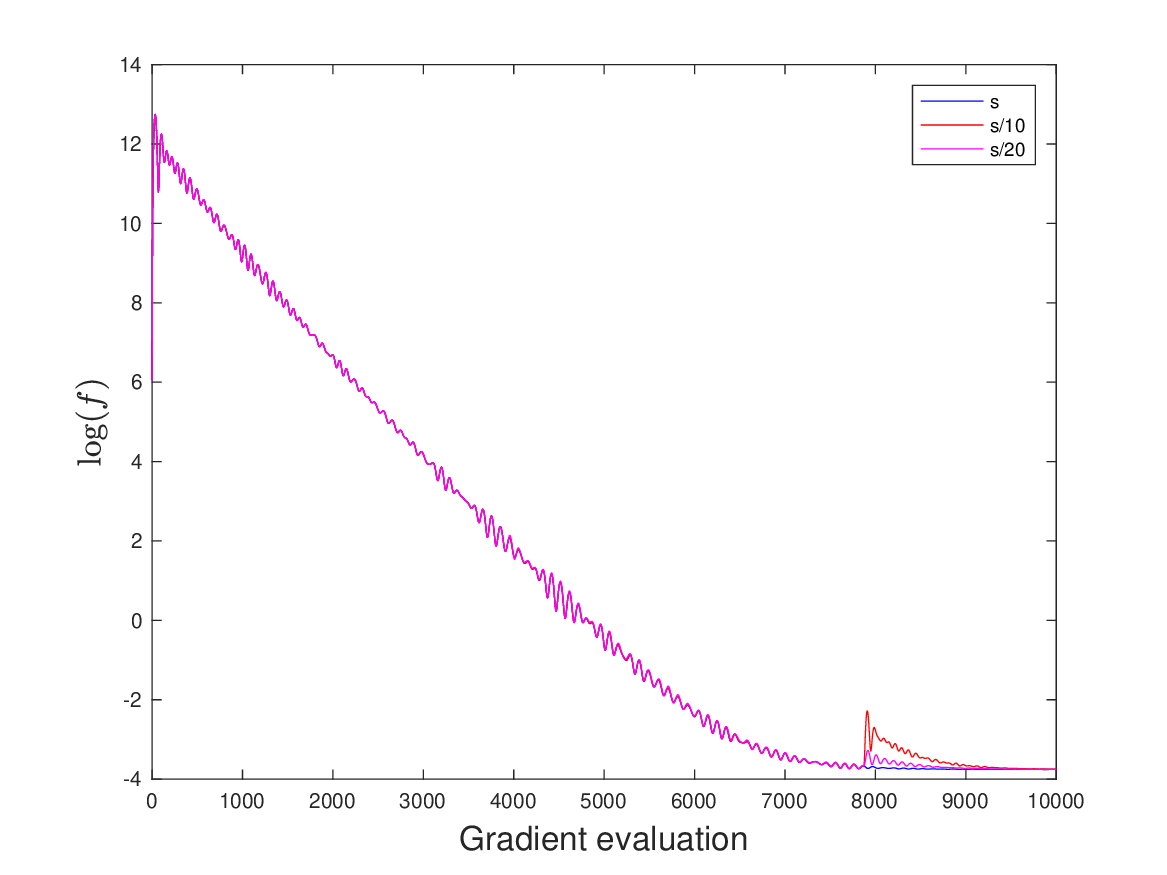}
\end{minipage}
\begin{minipage}[t]{0.4\linewidth}
\centering
\includegraphics[width=\textwidth]{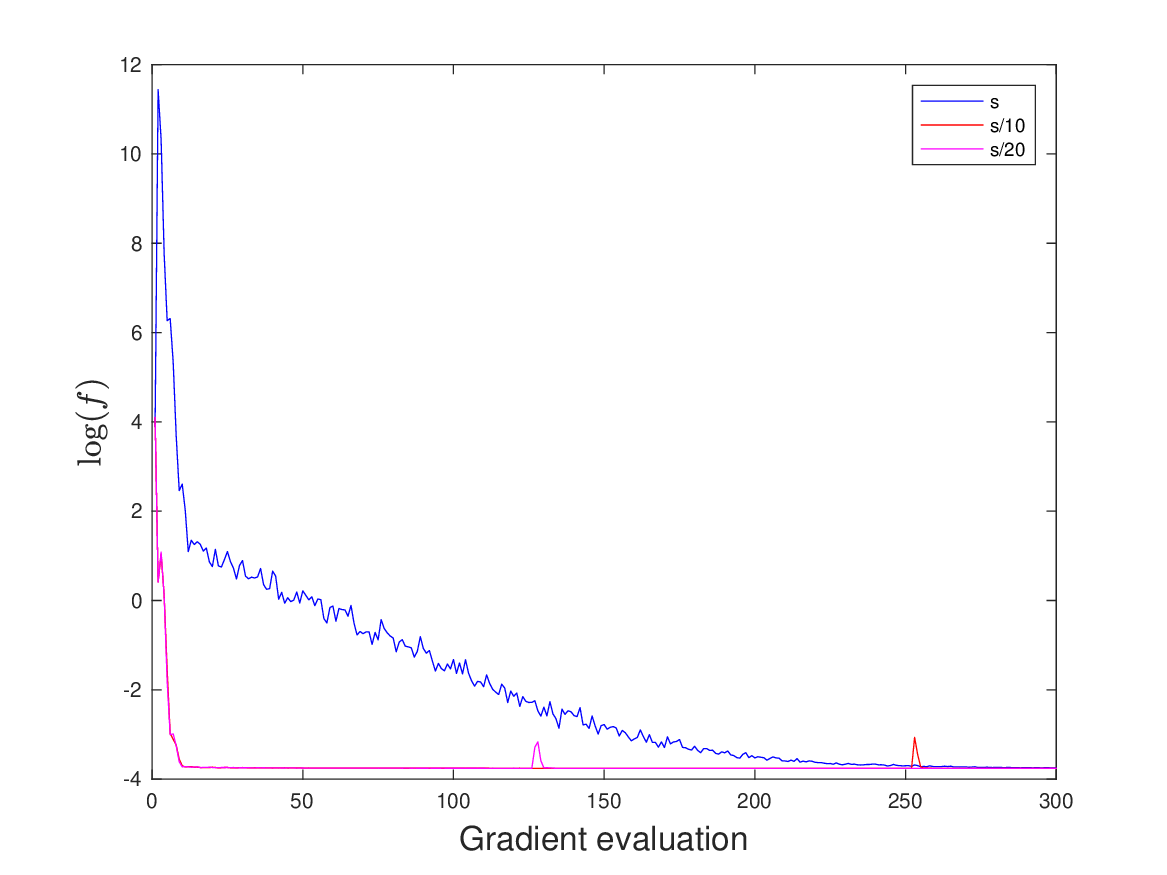}
\end{minipage}
\caption{Comparison on objective value for batch and randomized versions of RapGrad, when numbers of inner iterations are $s$, $s/10$, and $s/20$, respectively. Left Figure: Batch version. Right Figure: Randomized version.}
\label{fig 2}
\end{figure}

\begin{figure}[H]
\centering
\begin{minipage}[t]{0.4\linewidth}
\centering
\includegraphics[width=\textwidth]{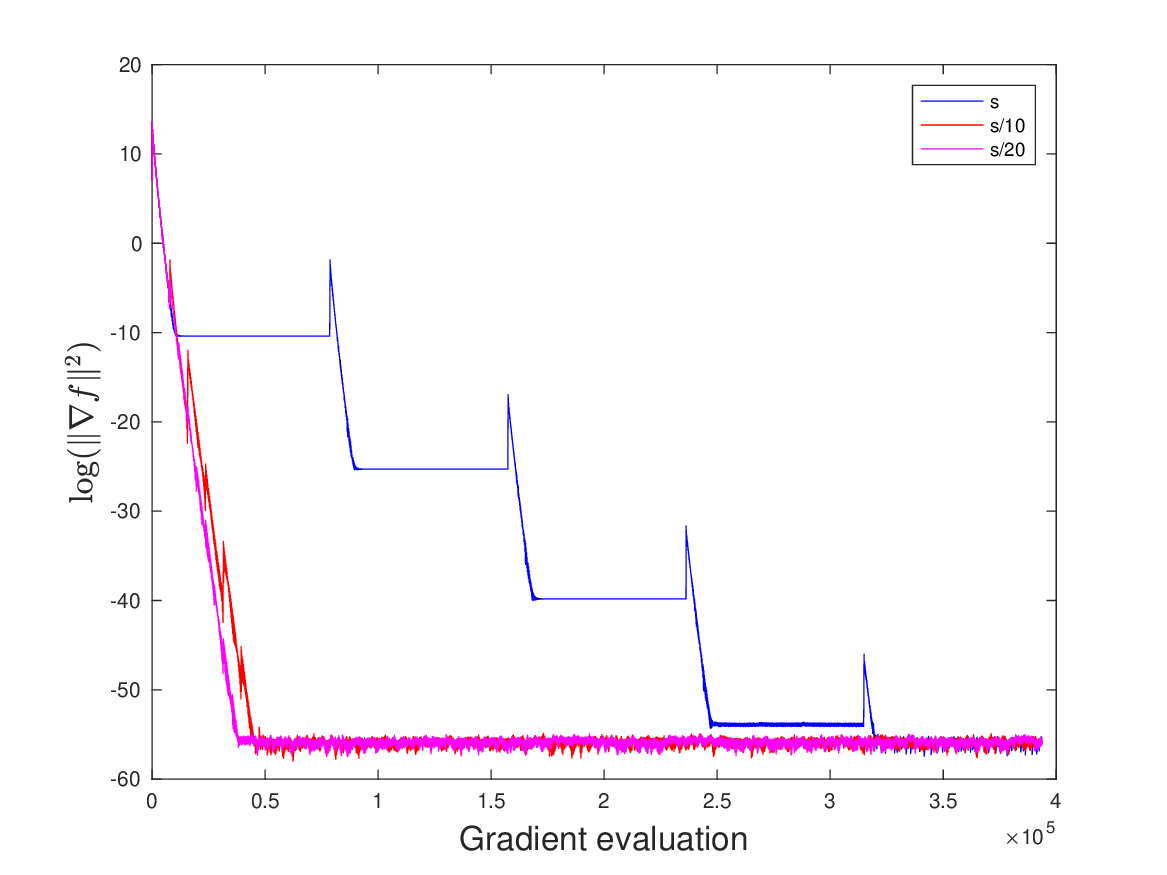}
\end{minipage}
\begin{minipage}[t]{0.4\linewidth}
\centering
\includegraphics[width=\textwidth]{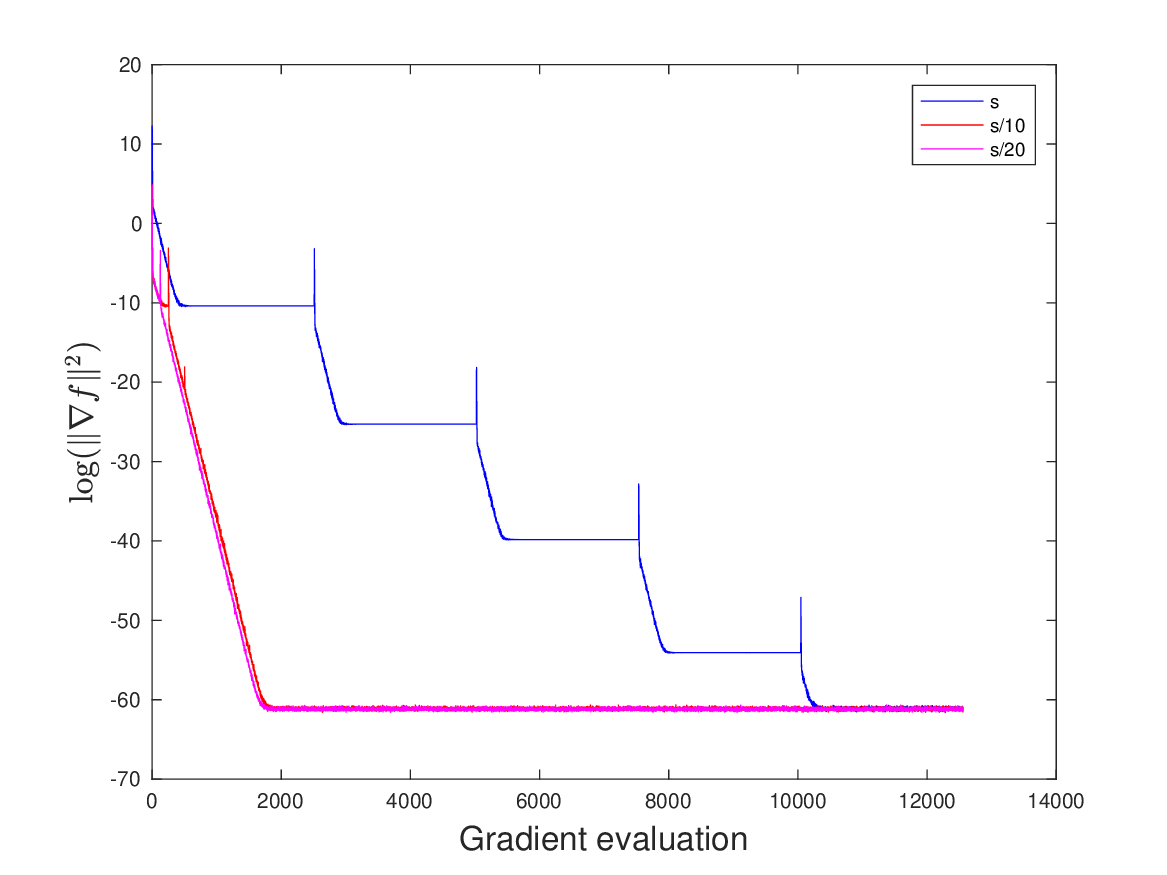}
\end{minipage}
\caption{Comparison on square of gradient norm $\|\n f\|^2$ for batch and randomized versions of RapGrad when numbers of inner iterations are $s$, $s/10$, and $s/20$, respectively. Left Figure: Batch version. Right Figure: Randomized version.}
\label{fig 3}
\end{figure}
Inspired by the above experiments, we have an efficient way to tune RapGrad to yield better performance. We first run RapGrad with several different numbers of inner iterations $s^{\prime}$, for instance $s^{\prime}=s$, $s^{\prime}=s/10$, $s^{\prime}=s/100$ , for a fixed number, say $\num{100}$, of passes through the dataset, then we use the best $s^{\prime}$  corresponding to the smallest norm of gradient as the actual $s$ for the tuned RapGrad. In the following table, we compare the RapGrad without tuning, tuned RapGrad, SVRG as well as AG on testing problems of different sizes, with stoping criteria $\|\n f\|^2< \num{e-10}$ and maximal pass $\num{3e4}$. The table shows that this simple tuning technique is able to bring huge performance improvement. An interesting observation is that RapGrad without tuning is more likely to outperform SVRG when $n$ is large relative to $m$.
\begin{table}[H]  
\centering
\begin{tabular}{|c|c|c|c|c|c|}
\hline
&RapGrad&RapGrad\_tuned&SVRG& AG\\
 \hline
 $m=1000$, $n=100$& $2850$&  $502$&$1143$&$\num{3e4}$\\ 
\hline
 $m=1000$, $n=300$& $4894$&  $874$&$5493$&$\num{3e4}$\\ 
  \hline
   $m=1000$, $n=500$& $11299$&  $1165$&$19029$&$\num{3e4}$\\ 
  \hline
 $m=800$, $n=100$& $3113$&  $559$&$1245$&$\num{3e4}$\\ 
  \hline
   $m=800$, $n=300$& $5467$&  $970$&$7743$&$\num{3e4}$\\ 
\hline
 $m=800$, $n=500$& $12673$&  $1290$&$\num{3e4}$&$\num{3e4}$\\ 
\hline
 $m=600$, $n=100$& $3735$&  $667$&$1752$&$\num{3e4}$\\
\hline
 $m=600$, $n=300$& $10978$&  $1137$&$13638$&$\num{3e4}$\\ 
\hline
 $m=600$, $n=500$& $14965$&  $490$&$\num{3e4}$&$\num{3e4}$\\ 
  \hline

\end{tabular}
\caption{Comparison on numbers of passes to the dataset with stopping criteria $\|\n f\|^2< \num{e-10}$ and maximal pass $\num{3e4}$.}\label{table1}
\end{table}
 
\subsection{Nonconvex multi-block optimization}\label{numerical 2}
We consider the following compressed sensing problem  to test the performance of RapDual: 
\[
\begin{aligned}
\min_{x_i\in X_i} &~ \tsum_{i=1}^m \mathcal{P}_{\lbd,\g,\e}(x_i)\\
\text{s.t. } &~\smn A_ix_i= b,
\end{aligned}
\]	
where each $x_i = (x_i^1,\ldots,x_i^{d_i})$ is a vector of dimension $d_i$, and $\mathcal{P}_{\lbd,\g,\e}(x_i) =\sum_{j=1}^{d_i}p_{\lbd,\g,\e}(x_i^j)$. Instead of using the $\l_1$ norm as the objective function, we replace it with the smoothed SCAD function, which is also capable of finding sparse solutions.
Since the smoothed SCAD function is separable in each component, we can easily identify an invertible $n\times n$ submatrix from $[A_1,A_2,\ldots,A_m]$. W.L.O.G, we assume the last block $A_m$ is invertible. If we multiply both sides of the linear equation $\sum_{i=1}^m A_ix_i = b$ by $A_m^{-1}$, we can reformulate 
the above problem into \eqref{c:problem_non}, which is ready to be solved by RapDual. 

The numerical experiments is performed on some randomly generated data sets of size $m=1001$, $n=100$. 
The first $1000$ coefficient matrices  $A_i$ , $1\leq i\leq 1000$, are of size $100\times 1$ and the last block $A_{1001}$ is an $100\times 100$ identity matrix.  $A_i$, $1\leq i\leq 1000$ are sparse matrices with sparsity level $0.1$, and all nonzero elements of $A_i$ and $200$ uniformly chosen  components of $\hat x$   are i.i.d from $N(0,1)$. The remaining variables of $\hat x$ are set to $0$ and $b = \sum_{i=1}^m A_i\hat x_i$. 
The parameters used in $p_{\lbd,\g,\e}$ 
are exactly the same as the first problem, i.e., $\e = \num{e-3}$, $\lbd = 2$, $\g = 4$.  From  Figure \ref{fig 4}, we can conclude that the randomized version 
converges faster than its batch counterpart, in terms of the number of primal block updates required to reduce the objective value and 
infeasibility. We also compare our randomized algorithm  with Algorithm 4 in \cite{hong2016convergence}, with $\rho = L^2, L^2/10, L^2/20$. Note  $\rho = L^2$ only guarantees asymptotic  convergence.
The results in  Figure \ref{fig 7} show that our algorithm can reduce the objective value faster than ADMM. As for the feasibility, both algorithms yield solutions that have quite tiny constraint violation.
It is worth noting that the converge of ADMM can very much depend on the update order ~\cite{ChenHeYeYuan13-1,SunLuoYe15-1} and also that
RapDual requires the computation of $A_m^{-1}$ in its current form.
Similar to RapGrad, if we reduce the number of inner iterations $s$ per subproblem  by a factor of $10$ or $20$, we obtain results in Figure \ref{fig 5} and Figure \ref{fig 6}. 
As we can see, the total number of primal block updates needed to yield a good solution, in terms of both objective value and feasibility, can be much smaller when inner loops are terminated early. 

\begin{figure}[H]
\centering
\begin{minipage}[t]{0.4\linewidth}
\centering
\includegraphics[width=\textwidth]{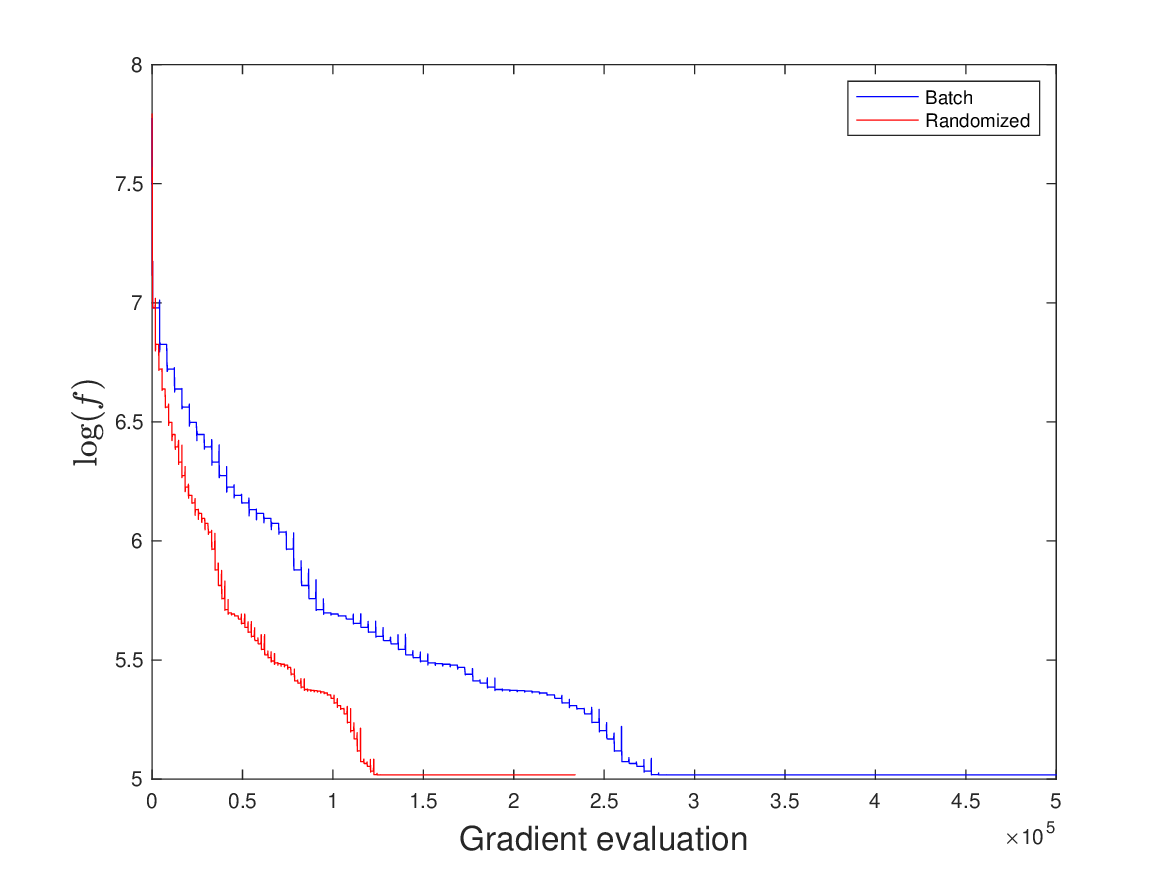}
\end{minipage}
\begin{minipage}[t]{0.4\linewidth}
\centering
\includegraphics[width=\textwidth]{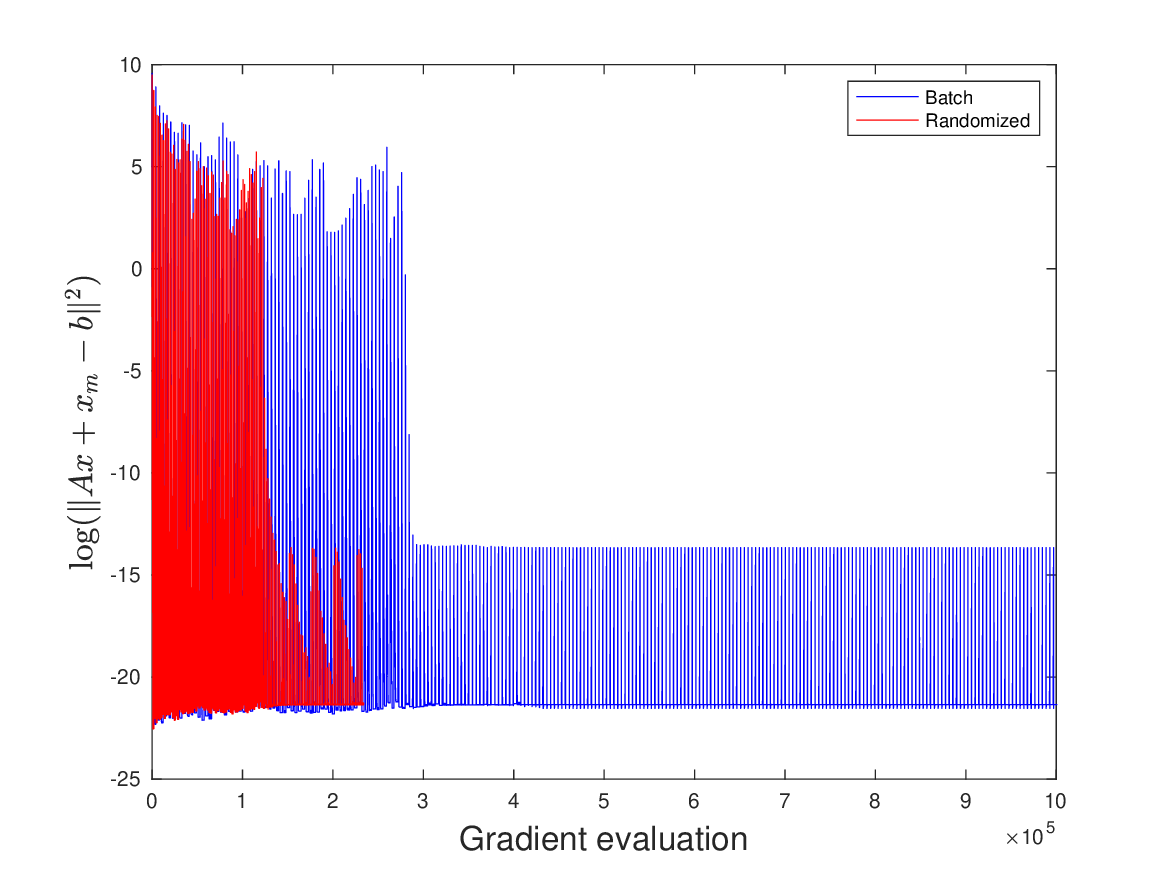}
\end{minipage}
\caption{Batch and randomized versions of Algorithm \ref{calg_non} on  compressed sensing problem with smoothed SCAD objective. Left Figure: Comparison on objective value $f$. Right Figure: Comparison on feasibility $\|\bfA\Bx+x_m-\bfb \|^2$.}\label{fig 4}
\end{figure}

\begin{figure}[H]   
\centering
\begin{minipage}[t]{0.4\linewidth}
\centering 
\includegraphics[width=\textwidth]{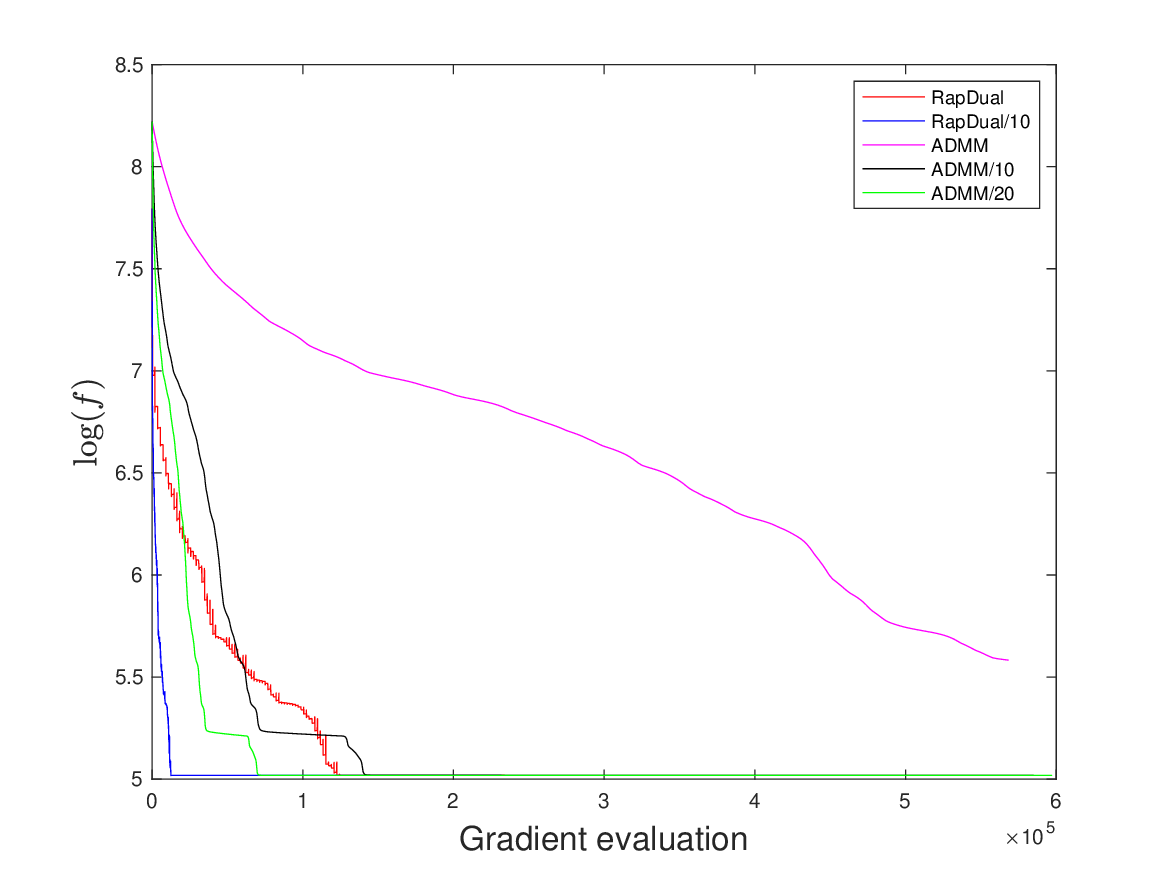}   
\end{minipage}
\begin{minipage}[t]{0.4\linewidth}
\centering
\includegraphics[width=\textwidth]{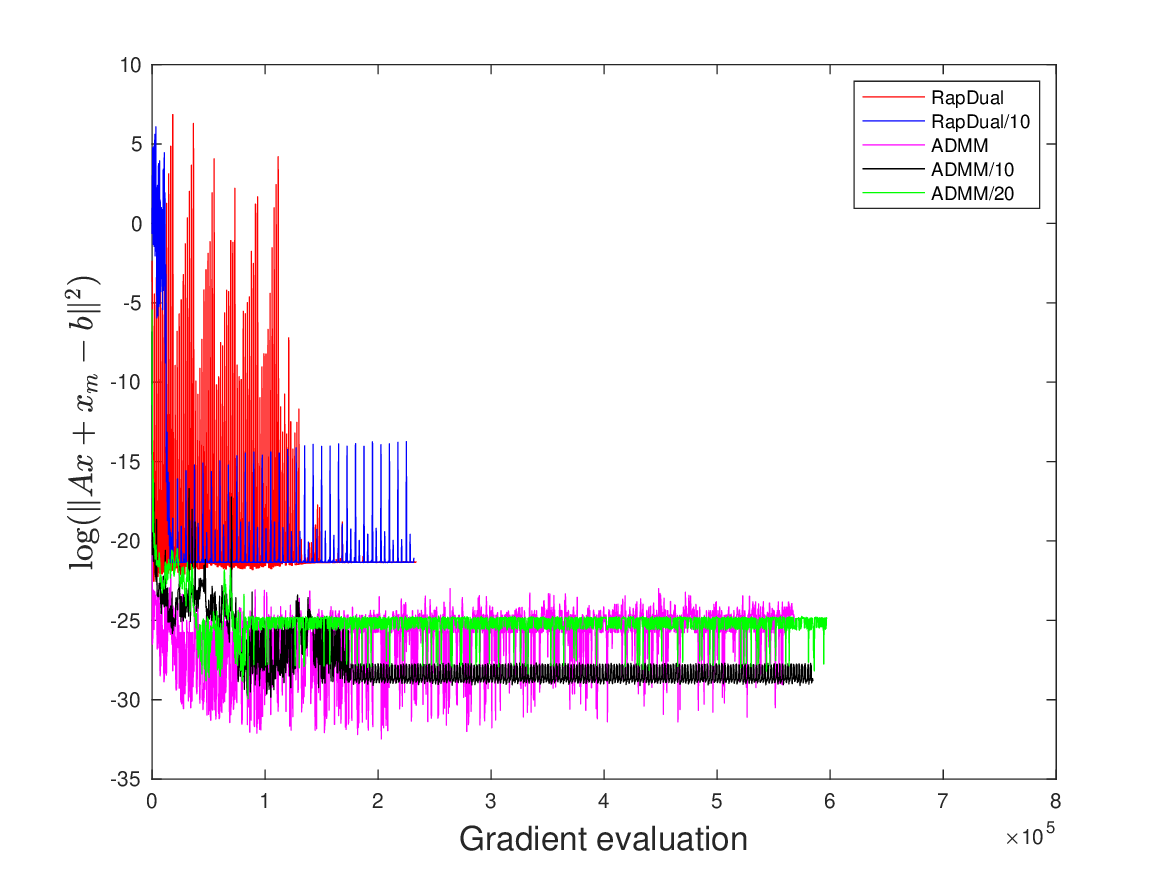}
\end{minipage}
\caption{Comparison on objective value $f$ and feasibility $\|\bfA\Bx+x_m-\bfb \|^2$ of   RapDual and ADMM  in \cite{hong2016convergence}. Left Figure: Comparison on objective value $f$ . Right Figure: Comparison on feasibility $\|\bfA\Bx+x_m-\bfb \|^2$ .}
\label{fig 7}
\end{figure}

\begin{figure}[H]
\centering  
\begin{minipage}[t]{0.4\linewidth}
\centering
\includegraphics[width=\textwidth]{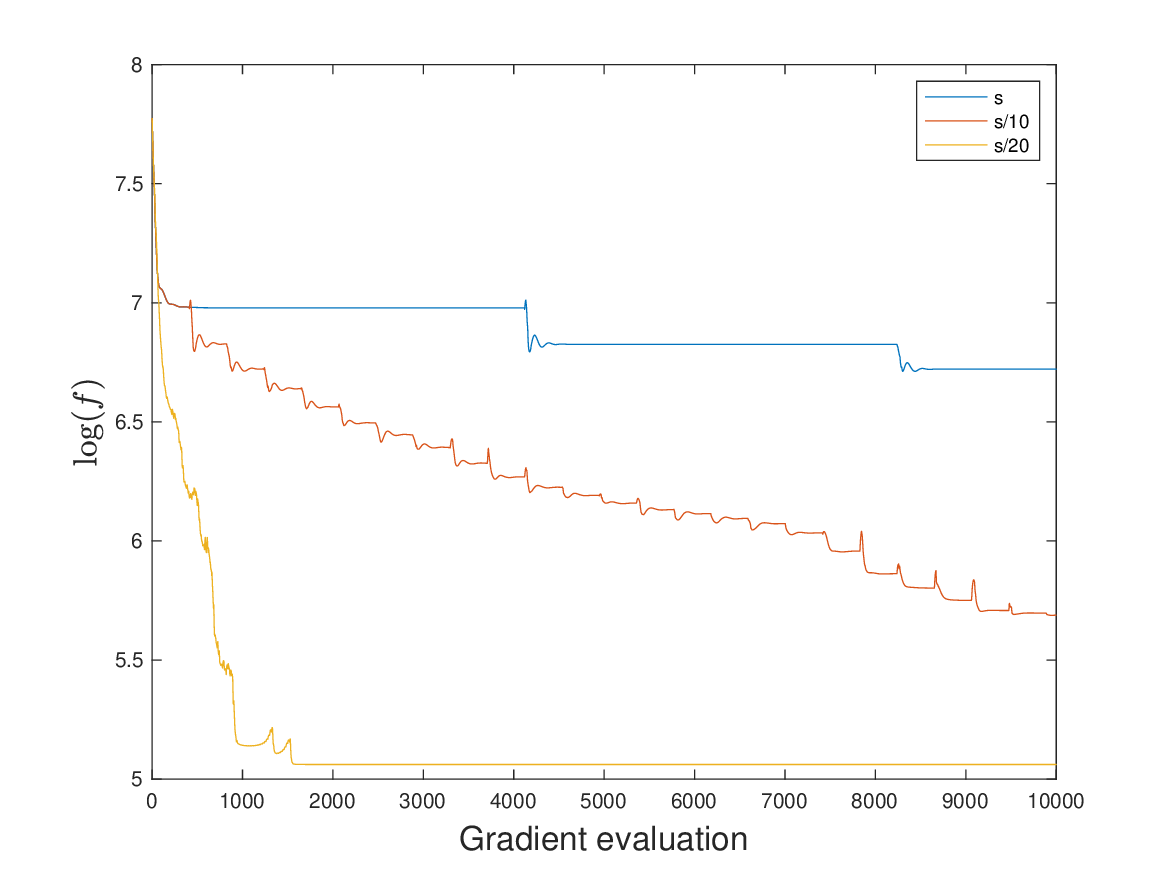}
\end{minipage}
\begin{minipage}[t]{0.4\linewidth}
\centering
\includegraphics[width=\textwidth]{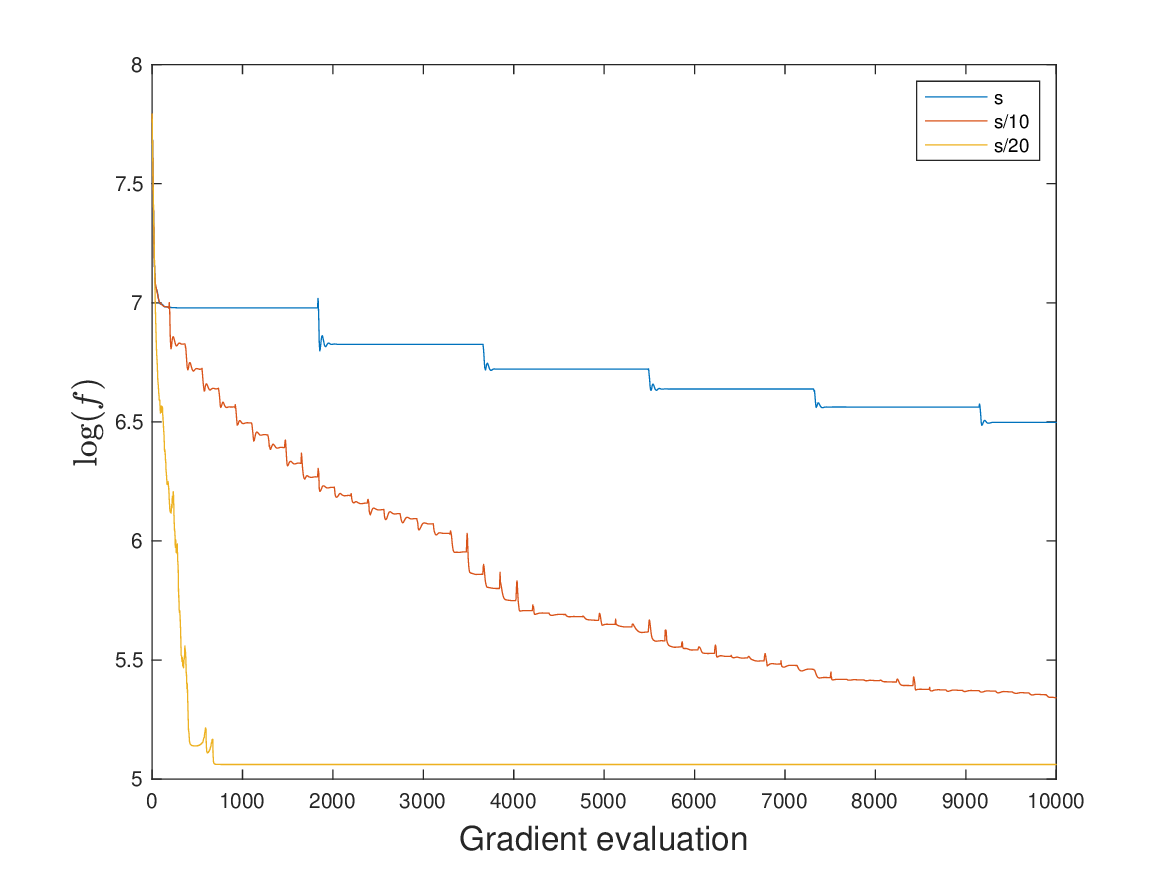}
\end{minipage}
\caption{Comparison on objective value $f$ of batch and randomized versions of RapDual when numbers of inner iterations are $s$, $s/10$, and $s/20$, respectively. Left Figure: Batch version. Right Figure: Randomized version.}
\label{fig 5} 
\end{figure}

\begin{figure}[H]   
\centering
\begin{minipage}[t]{0.4\linewidth}
\centering 
\includegraphics[width=\textwidth]{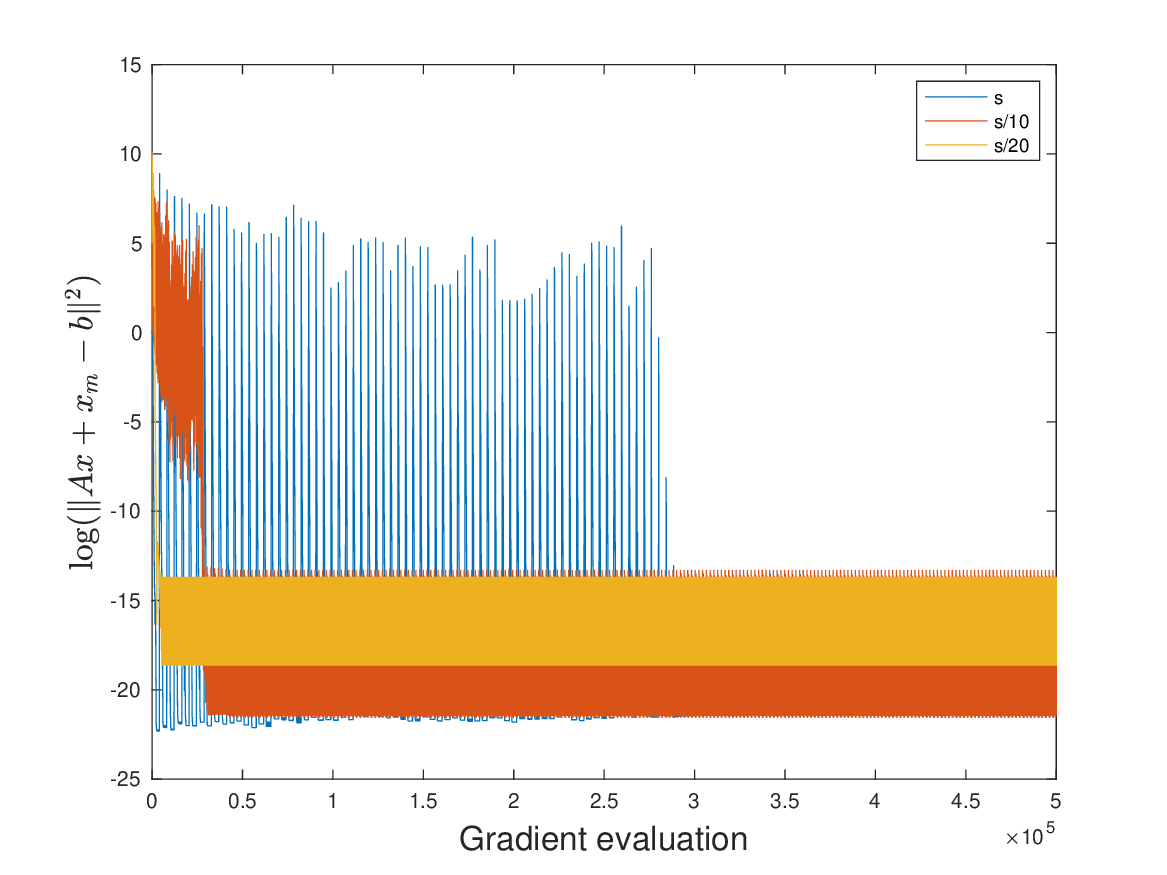}   
\end{minipage}
\begin{minipage}[t]{0.4\linewidth}
\centering
\includegraphics[width=\textwidth]{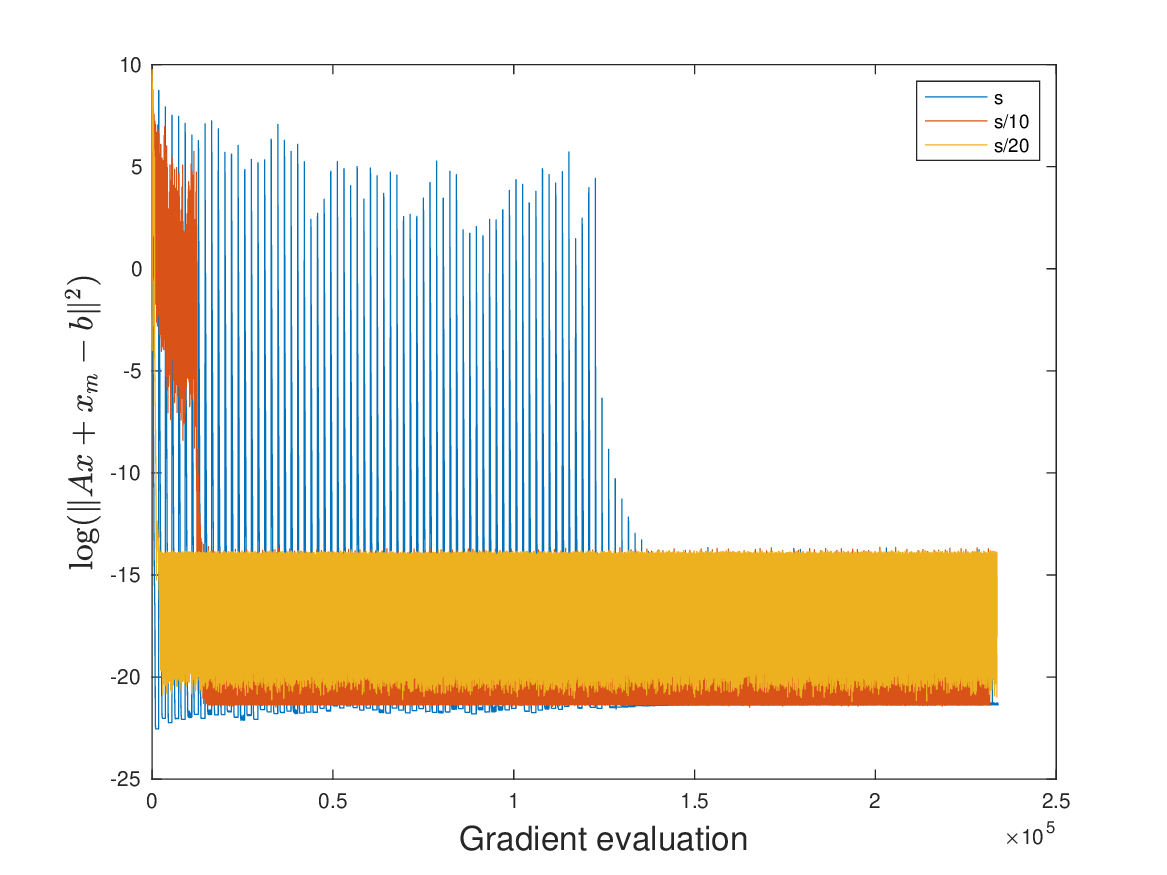}
\end{minipage}
\caption{Comparison on feasibility $\|\bfA\Bx+x_m-\bfb \|^2$ of batch and randomized versions of RapDual when numbers of inner iterations are $s$, $s/10$, and $s/20$, respectively. Left Figure: Batch version. Right Figure: Randomized version.}
\label{fig 6}
\end{figure}

\section{Concluding remarks}\label{conclusion}
In this paper, we propose a new randomized accelerated proximal-gradient (RapGrad) method for solving nonconvex finite-sum problems \eqref{problem_non} and a new randomized primal-dual gradient (RapGrad) method for  nonconvex multi-block problems \eqref{c:problem_non}, respectively. We demonstrate that for problem \eqref{problem_non} with large condition number, our RapGrad has much better convergence rate, in terms of dependence on the large number $m$, than the state-of-art nonconvex SVRG or SAGA, as well as accelerated gradient method
for nonconvex optimization. Moreover, we show that our RapDual method incorporated with randomization techniques can significantly save the number of primal block updates up to a factor of $\sqrt{m}$ than the deterministic batch methods for solving  problems \eqref{c:problem_non}. The potential advantages of RapGrad and RapDual  are also demonstrated through our preliminary numerical experiments.

We observe that the main ideas of this RapDual, i.e., using proximal point method to tranform the nonconvex multi-block problem into a series of convex subproblems, and
using randomized dual method to them, can be applied for solving much more general multi-block optimization problems for which there does not exist an invertible block. In this more general case,
the saddle-point subproblems will only be strongly convex in the primal space, but not in the dual space. 
Therefore, the complexity of solving the subproblem will only be sublinear, and as a consequence, the overall complexity will be much worse
than $\O(1/\epsilon)$. It will be interesting to study how much benefit one can obtain by using randomized algorithms in this more general case.
We leave this as an interesting topic for future research.
It is also worth noting that the proposed RapGrad and RapDual implicitly assume that the parameter $\mu$ is known, or a (tight) upper bound of it can be obtained.
While the values of $\mu$ for the problems considered 
in our numerical experiments can be tightly estimated, it will be interesting to see if we can adaptively estimate the value of $\mu$ in these algorithms
applied to solve more general problems.

\bibliographystyle{siam} 
\bibliography{Nonconvex.bib}

\newpage
\appendix
\section{Proof of Lemma \ref{lemma_dist}}\label{pf_lemma2}
\begin{proof} 
By convexity of $\psi$ and optimality of $x^*$, we have 
\begin{align}\label{def_Q}
~Q_t:&=~\varphi(x^t)+\psi(x^*)+\la \n \psi(x^*), x^t-x^*\ra-\left[\varphi(x^*)+\sm (\psi_i(\hat{\ux}^t_i)+\la \n \psi_i(\hat{\ux}^t_i), x^*-\hat{\ux}^t_i\ra)\right]\nn \\
&\geq ~\varphi(x^t)+\psi(x^*)+\la \n \psi(x^*), x^t-x^*\ra-\left[\varphi(x^*)+\psi(x^*)\right]\nn \\
&=~\varphi(x^t)-\varphi(x^*)+\la \n \psi(x^*), x^t-x^*\ra\geq \la \n \varphi(x^*)+\n \psi(x^*),x^t-x^*\ra\geq 0.
\end{align}
For notation convenience, let $\Psi(x,z) :=\psi(x)-\psi(z)-\la \n\psi(z), x-z\ra$,
\begin{align} 
Q_t &= \varphi(x^t)-\varphi(x^*)+\la \sm \tilde y_i^t, x^t-x^*\ra+\d^t_1+\d^t_2,\label{Q_equal}\\
\d^t_1&:= \psi(x^*)-\la \n\psi(x^*),x^*\ra-\sm [\psi_i(\hat{\ux}^t_i)-\la \n \psi_i(\hat{\ux}^t_i), \hat{\ux}^t_i\ra +\la \n \psi_i(\hat{\ux}^t_i)-\n \psi_i(x^*), \tilde x^t\ra]\nn\\
&~=\sm[\tau_t \Psi(\ux^{t-1}_i,x^*) -(1+\tau_t)\Psi(\hat{\ux}^t_i,x^*) -
\tau_t\Psi({\ux}^{t-1}_i,\hat{\ux}^t_i)],\label{delta1}\\
\d^t_2&:=\sm \left[ \la \n \psi_i(\hat{\ux}^t_i)-\n \psi_i(x^*),\tilde x^t\ra  -\la \tilde y_i^t-\n \psi_i(x^*), x^t\ra  +\la \tilde y_i^t-\n \psi_i(\hat{\ux}^t_i), x^* \ra  \right].\nn
\end{align}
In view of \eqref{eqn:full2}, we have
\begin{align}\label{delta2}
\Eb_{i_t}\d^t_2&=\sm\Eb_{i_t}\left[  \la \n \psi_i(\hat{\ux}^t_i)-\n \psi_i(x^*),\tilde x^t\ra  -\la \tilde y_i^t-\n \psi_i(x^*), x^t\ra  +\la \tilde y_i^t-\n \psi_i(\hat{\ux}^t_i),x^* \ra  \right]\nn\\
&=\sm \Eb_{i_t}\la \tilde y_i^t-\n \psi_i(x^*), \tilde x^t- x^t\ra  .
\end{align}
Multiplying each $Q_t$ by a non-negative $\g_t$ and summing them up, we obtain
\begin{align}\label{relation_Q}
\Eb_s\left[\sk \g_tQ_t\right] \leq &\Eb_s\left\{\sk \g_t[\eta_t V_\varphi(x^*,x^{t-1})-(1+\eta_t)V_\varphi(x^*,x^t)-\eta_t V_\varphi(x^t,x^{t-1})]\right\}\nn\\
& ~~+\Eb_s \left\{\sk\smn\left[\g_t(1+\tau_t-\tfrac{1}{m})  \Psi({\ux}^{t-1}_i,x^*)-\g_t (1+\tau_t)\Psi({\ux}^{t}_i,x^*)\right]\right\}\nn\\
& ~~+\Eb_s\left\{\sk\smn  \g_t \left[ \tfrac{1}{m}\la \tilde y^t_i-\n \psi_i(x^*),x^t-x^t\ra -  \tau_t \Psi({\ux}^{t-1}_i,{\ux}^{t}_i))\right]\right\}\nn\\
&\leq \Eb_s\left[\g_1\eta_1 V_\varphi(x^*,x^{0})-(1+\eta_s)V_\varphi(x^*,x^s)\right]\nn\\
& ~~+\Eb_s\left\{\smn[\g_1(1+\tau_1-\tfrac{1}{m})\Psi({\ux}^{0}_i,x^*)-\g_s (1+\tau_s)\Psi({\ux}^{s}_i,x^*)]\right\}-\Eb_s[\sk \g_t\d_t],
\end{align}
where 
\begin{align}\label{def_Delta}
\d_t:&=\eta_t V_\varphi(x^t,x^{t-1})-\smn  \left[\tfrac{1}{m}\la \tilde y^t_i-\n \psi_i(x^*),\tilde x^t-x^t\ra-\tau_t\Psi({\ux}^{t-1}_i,{\ux}^{t}_i)\right]\nn\\
&=\eta_t V_\varphi(x^t,x^{t-1})-\sm  \la \tilde y^t_i-\n \psi_i(x^*),\tilde x^t-x^t\ra+\tau_t\Psi({\ux}^{t-1}_{i_t},{\ux}^{t}_{i_t}),
\end{align}
the first inequality follows from \eqref{Q_equal}, \eqref{delta1}, \eqref{opt_x}, \eqref{delta2} and Lemma \ref{lemma_full}, and the second inequality is implied by \eqref{eqn:ss2} and \eqref{eqn:ss3}.

By the definition of $\tilde x$ in \eqref{def_tx}, we have
\begin{align}\label{Delta_r1}
&\sm\la \tilde y^t_i-\n \psi_i(x^*), \tilde x^t-x^t\ra \nn\\
&=\sm[\la \tilde y^t_i-\n \psi_i(x^*), x^{t-1}-x^t\ra-\a_t\la \tilde y^t_i-\n \psi_i(x^*), x^{t-2}-x^{t-1}\ra]\nn\\
&=\sm[\la \tilde y^t_i-\n \psi_i(x^*), x^{t-1}-x^t\ra-\a_t\la \tilde y^{t-1}_i-\n \psi_i(x^*), x^{t-2}-x^{t-1}\ra-\a_t\la \tilde y^{t}_i-\tilde y^{t-1}_i,  x^{t-2}-x^{t-1}\ra]\nn\\
&=\sm[\la\tilde y^t_i-\n \psi_i(x^*),  x^{t-1}-x^t\ra-\a_t\la \tilde y^{t-1}_i-\n \psi_i(x^*), x^{t-2}-x^{t-1}\ra]-\a_t\la\n \psi_{i_t}(\ux_{i_t}^t)
- \n \psi_{i_t}(\ux_{i_t}^{t-1}),  x^{t-2}-x^{t-1}\ra\nn\\
&-\tsy{(1-\tfrac{1}{m})\a_t\la \n \psi_{i_{t-1}}(\ux_{i_{t-1}}^{t-2})-\n \psi_{i_{t-1}}(\ux_{i_{t-1}}^{t-1}), x^{t-2}-x^{t-1}\ra}.
\end{align} 
 From the relation \eqref{eqn:ss5} and the fact $x^{-1}=x^0$, we have
 \begin{align}\label{Delta_r2}
&\sk\g_t\sm[\la \tilde y^t_i-\n \psi_i(x^*), x^{t-1}-x^t\ra-\a_t\la \tilde y^{t-1}_i-\n \psi_i(x^*), x^{t-2}-x^{t-1}\ra]\nn\\
&=\g_s\sm\la \tilde y^s_i-\n \psi_i(x^*), x^{s-1}-x^s\ra\nn\\
&=\g_s\sm\la \n \psi_{i}(\ux_{i}^s) -\n \psi_i(x^*), x^{s-1}-x^s\ra+\g_s\smn(1-\tfrac{1}{m})\la \n \psi_{i}(\ux_{i}^s) -\n \psi_{i}(\ux_{i}^{s-1}), x^{s-1}-x^s \ra\nn\\
&=\g_s\sm\la\n \psi_{i}(\ux_{i}^s) -\n \psi_i(x^*),  x^{s-1}-x^s\ra+\g_s\left(1-\tfrac{1}{m}\right)\la  \n \psi_{i_s}(\ux_{i_s}^s)- \n \psi_{i_s}(\ux_{i_s}^{s-1}), x^{s-1}-x^s\ra.
\end{align}
Now we are ready to bound the last term in \eqref{relation_Q} as follows:
\begin{align}
\sum_{t=1}^s \g_t \d_t \myeqa &~ \sk\g_t[\eta_t V_\varphi(x^t,x^{t-1})-\sm  \la \tilde y^t_i-\n \psi_i(x^*),\tilde x^t-x^t\ra+\tau_t\Psi({\ux}^{t-1}_{i_t},{\ux}^{t}_{i_t})]\nn\\
\myeqb &~\sk\g_t\Big[\eta_t V_\varphi(x^t,x^{t-1})+\a_t\la \n \psi_{i_t}(\ux_{i_t}^t)- \n \psi_{i_t}(\ux_{i_t}^{t-1}),x^{t-2}-x^{t-1}\ra\nn\\
&\quad \quad \quad \quad+\tsy{(1-\tfrac{1}{m})}\a_t\la \n \psi_{i_{t-1}}(\ux_{i_{t-1}}^{t-2})-\n \psi_{i_{t-1}}(\ux_{i_{t-1}}^{t-1})\ra,x^{t-2}-x^{t-1}+\tau_t\Psi({\ux}^{t-1}_{i_t},{\ux}^{t}_{i_t})\Big]\nn\\ 
&+\g_s\sm\la \n \psi_{i}(\ux_{i}^{s})-\n \psi_i(x^*), x^{s-1}-x^s\ra-\g_s\tsy{(1-\tfrac{1}{m})}\la \n \psi_{i_s}(\ux_{i_s}^s)- \n \psi_{i_s}(\ux_{i_s}^{s-1}),x^{s-1}-x^s\ra\nn\\
\mygeqc &~\sk\g_t\Big[\tfrac{\eta_tr}{2} \|x^t-x^{t-1}\|^2+\a_t\la \n \psi_{i_t}(\ux_{i_t}^t)- \n \psi_{i_t}(\ux_{i_t}^{t-1}), x^{t-2}-x^{t-1}\ra\nn\\
&~+\tsy{(1-\tfrac{1}{m})}\a_t\la \n \psi_{i_{t-1}}(\ux_{i_{t-1}}^{t-2})-\n \psi_{i_{t-1}}(\ux_{i_{t-1}}^{t-1}), x^{t-2}-x^{t-1}\ra+\tfrac{\tau_t}{2\hat L}\|\n \psi_{i_t}({\ux}^{t}_{i_t})-\n \psi_{i_t}({\ux}^{t-1}_{i_t})\|^2\Big]\nn\\
&~-\tfrac{\g_s}{m}\smn\la x^{s-1}-x^s, \n \psi_{i}(\ux_{i}^{s})-\n \psi_i(x^*)\ra-\g_s\left(1-\tfrac{1}{m}\right)\la x^{s-1}-x^s, \n \psi_{i_s}(\ux_{i_s}^s)- \n \psi_{i_s}(\ux_{i_s}^{s-1})\ra,\nn\\\label{break}
\end{align}
where (a) follows from the definition $\d_t$ in \eqref{def_Delta}, (b) follows relations \eqref{Delta_r1} and \eqref{Delta_r2} and 
(c) follows from the fact that 
$V_\varphi(x^t,x^{t-1})\geq \tfrac{r}{2}\|x^{t}-x^{t-1}\|^2$,  $
\Psi({\ux}^{t-1}_{i_t},{\ux}^{t}_{i_t})\geq \tfrac{1}{2\hat L}\|\n \psi_{i_t}({\ux}^{t-1}_{i_t})-\n \psi_{i_t}({\ux}^{t}_{i_t})\|^2$.

By properly regrouping the terms on the right hand side of \eqref{break}, we have
\begin{align*}
\tsum_{t=1}^s \g_t \d_t \geq &~\g_s\tsy{\left[\tfrac{\eta_s r}{4}\|x^s-x^{s-1}\|^2-\sm\la \n \psi_{i}(\ux_{i}^s)-\n \psi_i(x^*), x^{s-1}-x^s\ra\right]}\\
&~+\g_s\tsy{\big[\tfrac{\eta_s r}{4}\|x^s-x^{s-1}\|^2-(1-\tfrac{1}{m})\la \n \psi_{i_s}(\ux_{i_s}^{s})-\n \psi_{i_s}(\ux_{i_s}^{s-1}), x^{s-1}-x^s\ra}+\tsy{\tfrac{\tau_s}{4\hat L}\|\n\psi_{i_s}({\ux}^{s}_{i_s})-\n\psi_{i_s}({\ux}^{s-1}_{i_s})\|^2\big]}\\
&~+\skt \g_t\tsy{\left[\a_t \la \n \psi_{i_t}(\ux_{i_{t}}^{t})-\n \psi_{i_t}(\ux_{i_{t}}^{t-1}), x^{t-2}-x^{t-1},\ra+  \tfrac{\tau_t}{4\hat L}\|\n \psi_{i_t}({\ux}^{t}_{i_t})-\n \psi_{i_t}({\ux}^{t-1}_{i_t})\|^2\right]}\\
&~+\skt \tsy{\big[\g_t(1-\tfrac{1}{m})\a_t\la \n \psi_{i_{t-1}}(\ux_{i_{t-1}}^{t-1})-\n \psi_{i_{t-1}}(\ux_{i_{t-1}}^{t-2})}, x^{t-2}-x^{t-1}\ra\\
&\quad\quad\quad\quad+\tsy{\tfrac{\tau_{t-1}\g_{t-1}}{4\hat L}\|\n \psi_{i_{t-1}}({\ux}^{t-1}_{i_{t-1}})-\n \psi_{i_{t-1}}({\ux}^{t-2}_{i_{t-1}})\|^2\big]}+\skt \tfrac{\g_{t-1}\eta_{t-1} r}{2}\|x^{t-1}-x^{t-2}\|^2 \\
\mygeqa &~\g_s\tsy{\left[\tfrac{\eta_s r}{4}\|x^s-x^{s-1}\|^2-\sm\la \n \psi_{i}(\ux_{i}^s)-\n \psi_i(x^*),  x^{s-1}-x^s\ra\right]}\\
&~+\g_s\tsy{\left(\tfrac{\eta_sr}{4}-\tfrac{(m-1)^2\hat L}{m^2\tau_s}\right)\|x^s-x^{s-1}\|^2}+\skt\tsy{\left(\tfrac{\g_{t-1}\eta_{t-1}r}{2}-\tfrac{\g_t\a_t^2\hat L}{\tau_t}-\tfrac{(m-1)^2\g_t^2\a_t^2\hat L}{m^2\g_{t-1}\tau_{t-1}}\right)\|x^s-x^{s-1}\|^2}\\
\mygeqb &~\g_s\tsy{\left[\tfrac{\eta_s r}{4}\|x^s-x^{s-1}\|^2-\sm\la \n \psi_{i}(\ux_{i}^s)-\n \psi_i(x^*), x^{s-1}-x^s\ra\right]},
\end{align*}
where (a) follows from the simple relation that $b\la u, v\ra-a\|v\|^2/2 \le b^2\|u\|^2/(2a)$, $\forall a >0$ and  (b) follows from \eqref{eqn:ss1}, \eqref{eqn:ss4} and \eqref{eqn:ss5}. By using the above inequality, \eqref{def_Q} and \eqref{relation_Q}, we obtain
\begin{align}\label{final_Q}
	0\leq & ~\Eb_s\left[\g_1\eta_1 V_\varphi(x^*,x^{0})-\g_s(1+\eta_s)V_\varphi(x^*,x^s)\right]+\g_s\Eb_s \left[\sm\la \n \psi_{i}(\ux_{i}^s)-\n \psi_i(x^*), x^{s-1}-x^s\ra \right]\nn\\
&-\tsy{\tfrac{\g_s\eta_s r}{4}\Eb_s\|x^s-x^{s-1}\|^2}+\Eb_s\left\{\smn[\g_1(1+\tau_1-\tfrac{1}{m})\Psi({\ux}^{0}_i,x^*)-\g_s (1+\tau_s)\Psi({\ux}^{s}_i,x^*)]\right\}\nn\\
\myleqa &~\Eb_s[\g_1\eta_1 V_\varphi(x^*,x^{0})-\g_s(1+\eta_s)V_\varphi(x^*,x^s)]-\tsy{\tfrac{\g_s\eta_s r}{4}\Eb_s\|x^s-x^{s-1}\|^2}\nn\\
&+\Eb_s\left\{\smn[\g_1(1+\tau_1-\tfrac{1}{m})\Psi({\ux}^{0}_i,x^*)-\tfrac{\g_s (1+\tau_s)}{2}\Psi({\ux}^{s}_i,x^*)]\right\}\nn\\
&-\g_s\sm \Eb_s \left[\tfrac{m(1+\tau_s)}{4\hat L}\|\n \psi_i({\ux}^{s}_i)-\n \psi_i(x^*) \|^2-\la \n \psi_{i}(\ux_{i}^s)-\n \psi_i(x^*), x^{s-1}-x^s\ra\right]\nn\\
\myleqb &~\Eb_s[\g_1\eta_1 V_\varphi(x^*,x^{0})-\g_s(1+\eta_s)V_\varphi(x^*,x^s)]-\g_s\left[\tfrac{\eta_s r}{4}-\tfrac{\hat L}{m(1+\tau_s)}\right]\Eb_s\|x^s-x^{s-1}\|^2\nn\\
&+\smn\Eb_s[\g_1(1+\tau_1-\tfrac{1}{m})\Psi({\ux}^{0}_i,x^*)-\tfrac{\g_s (1+\tau_s)}{2}\Psi({\ux}^{s}_i,x^*)]\nn\\
\myleqc &~\Eb_s\left[\g_1\eta_1 V_\varphi(x^*,x^{0})-\g_s(1+\eta_s)V_\varphi(x^*,x^s)\right]\nn\\
&+\tfrac{\g_1[(1+\tau_1)-\tfrac1m]\hat L}{2}\smn\Eb_s \|{\ux}^{0}_i-x^*\|^2  -\tfrac{\mu\g_s (1+\tau_s)}{4}\smn\Eb_s \|{\ux}^{s}_i-x^*\|^2,
\end{align}
where (a) follows from $
\Psi({\ux}^{0}_i,x^*)\geq \tfrac{1}{2\hat L}\|\n \psi_{i_t}({\ux}^{0}_i)-\n \psi_{i_t}(x^*)\|^2$;
(b) follows from the simple relation that $b\la u, v\ra-a\|v\|^2/2 \le b^2\|u\|^2/(2a)$, $\forall a >0$ and  
(c) follows from \eqref{eqn:ss6}, strong convexity of $\psi_i$ and Lipschitz continuity of $\n\psi_i$. This completes the proof.
\end{proof}

\section{Proof of Lemma \ref{lemma 5}}\label{pf_lemma5}
\begin{proof}
For any $t\geq 1$, since $(\Bx^*, y^*)$ is a saddle point of \eqref{eqn:saddle}, we have 
\begin{align*}
\psi(\hat{\Bx}^t) - \psi(\Bx^*) + \la \bfA\hat{\Bx}^t-\bfb ,y^*\ra - \la \bfA\Bx^*-\bfb , y^t\ra+  h(y^*)-  h(y^t)\geq 0.
\end{align*}
For nonnegative $\g_t$, we further obtain
\begin{align}\label{eqn:Q1}
\Eb_s\left\{\sk \g_t \left[\psi(\hat{\Bx}^t) - \psi(\Bx^*) + \la \bfA\hat{\Bx}^t-\bfb ,y^*\ra - \la \bfA\Bx^*-\bfb , y^t\ra+  h(y^*)-  h(y^t)\right]\right\}\geq 0.
\end{align}
According to optimality conditions of \eqref{c:xfull} and \eqref{c:algo2} respectively, and strongly convexity of $\psi$ and  $  h$ we obtain
\begin{align*}
\psi(\hat{\Bx}^t) - \psi(\Bx^*) + \tfrac{\mu}{2}\|\Bx^*-\hat{\Bx}^t\|^2 + \la \bfA^{\top}y^t, \hat{\Bx}^t-\Bx^*\ra
&\le \tfrac{\eta_t}{2}\left[\|\Bx^*-\Bx^{t-1}\|^2 - \|\Bx^*-\hat{\Bx}^t\|^2 - \|\hat{\Bx}^t-\Bx^{t-1}\|^2\right],\\
  h(y^t)-  h(y^*)+\la -\bfA\tilde{\Bx}^t+\bfb ,y^t-y^*\ra &\le \tau_t V_{  h}(y^*,y^{t-1}) - (\tau_t+1) V_{  h}(y^*,y^{t})- \tau_t V_{  h}(y^t,y^{t-1}).
\end{align*}
Combining the above two inequalities with relation \eqref{eqn:Q1}, we have
\begin{align*}
&~\Eb_s\tsy{\left\{\sum_{t=1}^k\left[ \tfrac{\g_t\eta_t}{2}\|\Bx^*-\Bx^{t-1}\|^2 - \tfrac{\g_t(\eta_t+\mu)}{2}\|\Bx^*-\hat{\Bx}^t\|^2 - \tfrac{\g_t\eta_t}{2}\|\hat{\Bx}^t-\Bx^{t-1}\|^2\right]\right\}}\\
&~+ \Eb_s\tsy{\left\{\sk \g_t\left[\tau_t V_{  h}(y^*,y^{t-1})- (\tau_t+1) V_{  h}(y^*,y^{t})- {\tau_t}  V_{  h}(y^t,y^{t-1})\right]\right\}}\\
&~+ \Eb_s\left[\sk \g_t\la \bfA(\hat{\Bx}^t-\tilde{\Bx}^t),y^*-y^t\ra\right]\geq 0.
\end{align*}
Observe that for $t \ge 1$, 
\begin{align*}
\Eb_{i_t} \left\{\la \bfA(\hat{\Bx}^t-\tilde{\Bx}^t),y^*\ra \right\}
= \Eb_{i_t} \left\{\la \bfA((m-1)\Bx^t - (m-2)\Bx^{t-1}-\tilde{\Bx}^t),y^*\ra \right\}.
\end{align*}
Applying this and the results \eqref{c:full2}, \eqref{c:full3} in Lemma \ref{clemma:full}, we further have
\begin{align}\label{eqn:Q2}
0\leq &~\Eb_s\tsy{\left\{\sk \left[\tfrac{\g_t((m-1)\eta_t+(m-2)\mu)}{2}\|\Bx^*-\Bx^{t-1}\|^2 - \tfrac{(m-1)\g_t(\eta_t+\mu)}{2}\|\Bx^*-\Bx^t\|^2\right]\right\}}\nn\\
&~+ \tsy{\Eb_s\left\{\sk \left[ \g_t\tau_t   V_{  h}(y^*,y^{t-1})- {\g_t(\tau_t+1)}   V_{  h}(y^*,y^{t})\right]\right\}
+ \Eb_s\left\{\sk \g_t \d_t\right\}}\nn\\
\le &~
\tsy{\Eb_s\left[\tfrac{\g_1((m-1)\eta_1+(m-2)\mu)}{2}\|\Bx^*-\Bx^0\|^2 - \tfrac{(m-1)\g_s(\eta_s+\mu)}{2}\|\Bx^*-\Bx^s\|^2\right]}\nn\\
&~+ \tsy{\Eb_s[\g_1\tau_1  V_{  h}(y^*,y^{0}) - \g_s(\tau_s+1)  V_{  h}(y^*,y^s)]
+ \Eb_s[\sk \g_t \d_t]},
\end{align}
where
\[\tsy{\d_t = -\tfrac{(m-1)\eta_t}{2}\|x_{i_t}^t-x_{i_t}^{t-1}\|^2 - \tau_t  V_{  h}(y^{t},y^{t-1})+ \la \bfA ((m-1)\Bx^t - (m-2)\Bx^{t-1}-\tilde{\Bx}^t),y^*-y^t\ra}.\]
and the second inequality follows from \eqref{c:ss3} and \eqref{c:ss4}.

By \eqref{c:ss1} and the definition of $\tilde{x}^t$ in \eqref{c:algo1} we have:
\begin{align}
\sk \g_t \d_t =& \tsy{\sk \left[-\tfrac{(m-1)\g_t\eta_t}{2}\|x_{i_t}^t-x_{i_t}^{t-1}\|^2 - \g_t\tau_t  V_{  h}(y^{t},y^{t-1})\right]}\nn\\
&~+\tsy{\sk \g_t (m-1) \la \bfA (\Bx^t - \Bx^{t-1}),y^*-y^t\ra
-\sk \g_t (m-1)\tilde{\a}_t \la \bfA (\Bx^{t-1} - \Bx^{t-2}),y^*-y^{t-1}\ra}\nn\\
&~-\tsy{\sk \g_t (m-1)\tilde{\a}_t \la \bfA (\Bx^{t-1} - \Bx^{t-2}),y^{t-1}-y^t\ra}\nn\\
=& \tsy{\sk \left[-\tfrac{(m-1)\g_t\eta_t}{2}\|x_{i_t}^t-x_{i_t}^{t-1}\|^2 - \g_t\tau_t  V_{  h}(y^{t},y^{t-1})\right]}+\g_s (m-1) \la \bfA (\Bx^s - \Bx^{s-1}),y^*-y^s\ra\nn\\
&~-\sk \g_t (m-1)\tilde{\a}_t \la \bfA (\Bx^{t-1} - \Bx^{t-2}),y^{t-1}-y^t\ra,\label{break2}
\end{align}
where the second equality follows from \eqref{c:ss2} and the fact that $x^0 = x^{-1}$.

Since $\la \bfA (\Bx^{t-1} - \Bx^{t-2}),y^{t-1}-y^t\ra
=\la \bfA _{t-1}(x_{i_{t-1}}^{t-1} - x_{i_{t-1}}^{t-2}),y^{t-1}-y^t\ra\leq   \|\bfA _{i_{t-1}}\|\|x_{i_{t-1}}^{t-1} - x_{i_{t-1}}^{t-2}\|\|y^t-y^{t-1}\|
$ and $V_{  h}(y^{t},y^{t-1})\geq \tfrac{\bar\mu}{2}\|y^t-y^{t-1}\|^2$, from \eqref{break2} we have 
\begin{align*}
\sk \g_t \d_t\leq 
& \tsy{\sk \left[-\tfrac{(m-1)\g_t\eta_t}{2}\|x_{i_t}^t-x_{i_t}^{t-1}\|^2 - \tfrac{g_t\tau_t\bar\mu}{2} \| y^{t},y^{t-1}\|^2\right]}+\g_s (m-1) \la \bfA (\Bx^s - \Bx^{s-1}),y^*-y^s\ra\nn\\
&~-\sk \g_t (m-1)\tilde{\a}_t \|\bfA _{i_{t-1}}\|\|x_{i_{t-1}}^{t-1} - x_{i_{t-1}}^{t-2}\|\|y^t-y^{t-1}\|\\
\myeqa&~\tsy{\g_s (m-1) \la \bfA (\Bx^s - \Bx^{s-1}),y^*-y^s\ra-\tfrac{(m-1)\g_s\eta_s}{2}\|x_{i_s}^s-x_{i_s}^{s-1}\|^2}\\
&~+\tsy{\skt \big[\g_t (m-1)\tilde{\a}_t \|\bfA _{i_{t-1}}\|\|x_{i_{t-1}}^{t-1} - x_{i_{t-1}}^{t-2}\|\|y^t-y^{t-1}\|}\\
&\quad\quad\quad\quad -\tsy{\tfrac{(m-1)\g_{t-1}\eta_{t-1}}{2}\|x_{i_{t-1}}^{t-1}-x_{i_{t-1}}^{t-2}\|^2 - \tfrac{\bar \mu\g_t\tau_t}{2} \|y^{t}-y^{t-1}\|^2\big]}\\
\myleqb &~\tsy{\g_s (m-1) \la \bfA (\Bx^s - \Bx^{s-1}),y^*-y^s\ra-\tfrac{(m-1)\g_s\eta_s}{2}\|x_{i_s}^s-x_{i_s}^{s-1}\|^2}\\
&~+\tsy{\skt \left(\tfrac{\g_t^2 (m-1)^2\tilde{\a}_t^2 \bar{A}^2}{2(m-1)\g_{t-1}\eta_{t-1}}- \tfrac{\bar\mu\g_t\tau_t}{2}\right)\|y^t-y^{t-1}\|^2}\\
\myeqc &~\tsy{\g_s (m-1) \la \bfA (\Bx^s - \Bx^{s-1}),y^*-y^s\ra-\tfrac{(m-1)\g_s\eta_s}{2}\|x_{i_s}^s-x_{i_s}^{s-1}\|^2},
\end{align*}
where (a) follows from regrouping the terms;
(b) follows from the definition $\bar{A} = \max_{i\in [m-1]} \|\bfA _{i}\|$ and the simple relation that $b\la u, v\ra-a\|v\|^2/2 \le b^2\|u\|^2/(2a)$, $\forall a >0$;
and (c) follows from \eqref{c:ss2} and \eqref{c:ss5}.

By combining the relation above with \eqref{eqn:Q2}, we obtain 
\begin{align}\label{c:gap3}
0\leq 
&~\tsy{\Eb_s\left[\tfrac{\g_1((m-1)\eta_1+(m-2)\mu)}{2}\|\Bx^*-\Bx^0\|^2 - \tfrac{(m-1)\g_s(\eta_s+\mu)}{2}\|\Bx^*-\Bx^s\|^2\right]}\nn\\
&~+ \tsy{\Eb_s\left[{\g_1\tau_1}V_{  h}(y^*,y^{0}) -{\g_s(\tau_s+1)}V_{  h}(y^*,y^s)\right]}\nn\\
&~+ \tsy{\Eb_s\left[\g_s (m-1) \la \bfA (\Bx^s - \Bx^{s-1}),y^*-y^s\ra-\tfrac{(m-1)\g_s\eta_s}{2}\|x_{i_s}^s-x_{i_s}^{s-1}\|^2\right]}.
\end{align} 
Notice the fact that 
\begin{align}
 &~\tsy{\Eb_s\left[\tfrac{\g_s(\tau_s+1)}{2}V_{  h}(y^*,y^s)-\g_s (m-1) \la \bfA (\Bx^s - \Bx^{s-1}),y^*-y^s\ra+\tfrac{(m-1)\g_s\eta_s}{2}\|x_{i_s}^s-x_{i_s}^{s-1}\|^2\right]}\nn\\
 &~ =\tsy{\Eb_s\left[\tfrac{\g_s(\tau_s+1)}{2}V_{  h}(y^*,y^s)-\g_s (m-1) \la \bfA (x_{i_s}^s - x_{i_s}^{s-1}),y^*-y^s\ra+\tfrac{(m-1)\g_s\eta_s}{2}\|x_{i_s}^s-x_{i_s}^{s-1}\|^2\right]}\nn\\
&~ \geq \tsy{\g_s\Eb_s\left[\tfrac{(\tau_s+1)\bar\mu}{4}\|y^*-y^s\|^2- (m-1)\bar{A}\|x_{i_s}^s - x_{i_s}^{s-1}\|\|y^*-y^s\|+\tfrac{(m-1)\eta_s}{2}\|x_{i_s}^s-x_{i_s}^{s-1}\|^2\right]}\nn\\
 &~ \geq \tsy{\g_s\Eb_s\left(\sqrt{\tfrac{(m-1)(\tau_s+1 )\bar\mu\eta_s}{2}}-(m-1)\bar{A}\right)\|x_{i_s}^s - x_{i_s}^{s-1}\|\|y^*-y^s\|\geq 0}.\label{c:gap4}
\end{align}
In view of \eqref{c:gap3} and \eqref{c:gap4}, we complete the proof.
\end{proof}

\end{document}